\documentclass[reqno]{amsart}

\usepackage{amsmath}
\usepackage{amssymb}
\usepackage{amsfonts}
\usepackage{amsthm}

\usepackage[utf8]{inputenc}
\usepackage[T1]{fontenc}

\usepackage{dsfont}

\usepackage[dvipsnames,svgnames]{xcolor}
\colorlet{MyBlue}{DodgerBlue!75!Black}

\usepackage{subfigure}

\usepackage{acronym}
\usepackage{latexsym}
\usepackage{paralist}
\usepackage{xspace}

\usepackage[authoryear,longnamesfirst,sort&compress]{natbib}

\usepackage{hyperref}
\hypersetup{
colorlinks=true,
linktocpage=true,
pdfstartview=FitH,
breaklinks=true,
pdfpagemode=UseNone,
pageanchor=true,
pdfpagemode=UseOutlines,
plainpages=false,
bookmarksnumbered,
bookmarksopen=false,
bookmarksopenlevel=1,
hypertexnames=true,
pdfhighlight=/O,
urlcolor=Maroon,linkcolor=MyBlue!60!black,citecolor=DarkGreen!70!black, 
pdftitle={},
pdfauthor={},
pdfsubject={},
pdfkeywords={},
pdfcreator={pdfLaTeX},
pdfproducer={LaTeX with hyperref}
}

\newcommand{\R}{\mathbb{R}}

\newcommand{\N}{\mathbb{N}}

\DeclareMathOperator*{\argmax}{arg\,max}
\DeclareMathOperator*{\argmin}{arg\,min}
\DeclareMathOperator{\bd}{bd}
\DeclareMathOperator{\bigoh}{\mathcal O}

\DeclareMathOperator{\cl}{cl}

\DeclareMathOperator{\ex}{\mathbb{E}}
\DeclareMathOperator{\exclude}{\backslash}

\DeclareMathOperator{\hess}{Hess}

\DeclareMathOperator{\intr}{int}

\DeclareMathOperator{\prob}{\mathbb{P}}

\DeclareMathOperator{\supp}{supp}
\DeclareMathOperator{\tr}{tr}

\DeclareMathOperator*{\union}{\bigcup}

\newcommand{\dd}{\,d}

\newcommand{\eps}{\varepsilon}
\newcommand{\from}{\colon}

\newcommand{\pd}{\partial}
\newcommand{\simplex}{\Delta}
\newcommand{\wilde}{\widetilde}

\newcommand{\abs}[1]{\left\lvert #1 \right\rvert}
\newcommand{\smallabs}[1]{\lvert #1 \rvert}
\newcommand{\norm}[1]{\left\| #1 \right\|}
\newcommand{\smallnorm}[1]{\| #1 \|}

\newcommand{\product}[2]{\left\langle #1,  #2 \right\rangle}
\newcommand{\smallproduct}[2]{\langle #1,  #2 \rangle}

\newcommand{\braket}[2]{\product{#1}{#2}}
\newcommand{\smallbraket}[2]{\smallproduct{#1}{#2}}

\newcommand{\defeq}{\equiv}

\newcommand{\dis}{\displaystyle}
\newcommand{\txs}{\textstyle}

\newcommand{\insum}{\sum\nolimits}

\usepackage{todonotes}
\usepackage{soul}
\setstcolor{red}
\sethlcolor{SkyBlue}

\theoremstyle{plain}
\newtheorem{theorem}{Theorem}
\newtheorem{corollary}[theorem]{Corollary}
\newtheorem*{corollary*}{Corollary}
\newtheorem{lemma}[theorem]{Lemma}
\newtheorem{proposition}[theorem]{Proposition}

\theoremstyle{definition}
\newtheorem{definition}[theorem]{Definition}
\newtheorem*{definition*}{Definition}
\newtheorem{assumption}{Assumption}
\newtheorem*{assumption*}{Assumption}

\theoremstyle{remark}
\newtheorem{remark}{Remark}
\newtheorem*{remark*}{Remark}
\newtheorem{example}{Example}

\numberwithin{equation}{section}
\numberwithin{theorem}{section}
\numberwithin{remark}{section}
\numberwithin{example}{section}

\newcommand{\play}{\mathcal{N}}
\newcommand{\act}{\mathcal{A}}
\newcommand{\pay}{u}
\newcommand{\payv}{v}
\newcommand{\strat}{\mathcal{X}}
\newcommand{\game}{\mathcal{G}}
\newcommand{\brep}{\mathsf{br}}

\newcommand{\eq}{x^{\ast}}

\DeclareMathOperator{\reg}{Reg}
\DeclareMathOperator{\erfc}{erfc}
\newcommand{\choice}{Q}
\newcommand{\fench}{F}
\newcommand{\temp}{\eta}
\newcommand{\depth}{\Omega}

\newcommand{\gibbs}{G}
\newcommand{\eucl}{\Pi}

\newcommand{\set}{\mathcal{S}}
\newcommand{\interval}{I}
\newcommand{\bvec}{e}

\newcommand{\intstrat}{\strat^{\circ}}
\newcommand{\intsimplex}{\simplex^{\!\circ}}

\newcommand{\one}{\mathds{1}}

\newcommand{\graph}{\mathcal{G}}
\newcommand{\vertices}{\mathcal{V}}
\newcommand{\edges}{\mathcal{E}}
\newcommand{\loss}{\ell}

\begin{document}


\title
[Learning in Games with Noisy Payoff Observations]
{On the robustness of learning in games with stochastically perturbed payoff observations}

\author[M.~Bravo]{Mario Bravo}
\address
[M.~Bravo]
{Universidad de Santiago de Chile, Departamento de Matem\'atica y Ciencia de la Computaci\'on, Av.Libertador Bernardo O’Higgins 3363, Santiago, Chile}
\email{\href{mailto:mario.bravo.g@usach.cl}{mario.bravo.g@usach.cl}}

\author[P.~Mertikopoulos]{Panayotis Mertikopoulos}
\address
[P.~Mertikopoulos]
{CNRS (French National Center for Scientific Research), LIG, F-38000 Grenoble, France\\
and
Univ. Grenoble Alpes, LIG, F-38000 Grenoble, France}
\email{\href{mailto:panayotis.mertikopoulos@imag.fr}{panayotis.mertikopoulos@imag.fr}}
\urladdr{\url{http://mescal.imag.fr/membres/panayotis.mertikopoulos}}

\thanks{%
The authors are greatly indebted to Roberto Cominetti for arranging the visit of the second author to the University of Chile and for his many constructive comments.
The authors would also like to express their gratitude to Mathias Staudigl for his many insightful comments and suggestions, to Bill Sandholm, Josef Hofbauer, and Yannick Viossat for helpful discussions, and to two anonymous referees for their detailed remarks and recommendations.}

\thanks{%
Part of this work was carried out during the authors' visit to the Hausdorff Research Institute for Mathematics at the University of Bonn in the framework of the Trimester Program ``Stochastic Dynamics in Economics and Finance'' and during the second author's visit to the University of Chile.
MB was partially supported by Fondecyt grant No. 11151003, and the N\'ucleo Milenio Informaci\'on y Coordinaci\'on en Redes ICM/FIC  RC130003.
PM was partially supported by the French National Research Agency (grant nos. NETLEARN--13--INFR--004 and GAGA--13--JS01--0004--01) and the French National Center for Scientific Research (grant no. REAL.NET--PEPS--JCJC--INS2I--2014.)}



\newacro{EW}{exponential weight}
\newacro{KKT}{Karush\textendash Kuhn\textendash Tucker}
\newacro{MDP}{Markov decision process}
\newacroplural{MDP}{Markov decision processes}
\newacro{SDE}{stochastic differential equation}
\newacro{ICT}{internally chain transitive}
\newacro{OMD}{online mirror descent}

\begin{abstract}
%
%
Motivated by the scarcity of accurate payoff feedback in practical applications of game theory, we examine a class of learning dynamics where players adjust their choices based on past payoff observations that are subject to noise and random disturbances.
First, in the single-player case (corresponding to an agent trying to adapt to an arbitrarily changing environment), we show that the stochastic dynamics under study lead to no regret almost surely, irrespective of the noise level in the player's observations.
In the multi-player case, we find that dominated strategies become extinct and we show that strict Nash equilibria are stochastically stable and attracting;
conversely, if a state is stable or attracting with positive probability, then it is a Nash equilibrium.
Finally, we provide an averaging principle for $2$-player games, and we show that in zero-sum games with an interior equilibrium, time averages converge to Nash equilibrium for any noise level.
\end{abstract}

\maketitle
\thispagestyle{empty}


\vspace{-2em}
\setcounter{tocdepth}{1}
\tableofcontents

\section{Introduction}
\label{sec:introduction}

A central question in game-theoretic learning is whether the outcome of a learning process (viewed here as a plausible model for the behavior of optimizing agents) is also justifiable from the point of view of rationality \textendash\ e.g. whether the process leads to a Nash equilibrium or a state where no dominated strategies are present.
In that regard, one of the most widely studied learning procedures is the \ac{EW} algorithm that was originally introduced by \cite{Vov90} and \cite{LW94} in the context of multi-armed bandit problems.
In a game-theoretic setting, the algorithm simply prescribes that players score their actions based on their cumulative payoffs and then assign choice probabilities proportionally to the exponential of each action's score.
As such, the \ac{EW} algorithm in continuous time \citep{Sor09} is described by the dynamics
\begin{equation}
\label{eq:EW}
\tag{EW}
\begin{aligned}
\dot y_{k\alpha}
	&= \payv_{k\alpha},
	\\
x_{k\alpha}
	&= \frac{\exp(y_{k\alpha})}{\insum_{\beta} \exp(y_{k\beta})},
\end{aligned}
\end{equation}
where, somewhat informally, $\payv_{k\alpha}$ denotes the payoff to the $\alpha$-th action of player $k$, $y_{k\alpha}$ is the action's performance score (cumulative payoff) over time and $x_{k\alpha}$ is the corresponding mixed strategy weight.

Given its long history and its link with the replicator dynamics of evolutionary game theory (discussed below), the rationality properties of \eqref{eq:EW} are also relatively well understood.
To name the most important ones:
\begin{inparaenum}
[\itshape a\upshape)]
\item
dominated strategies become extinct in the long run;
\item
limits of interior trajectories and stable rest points are Nash equilibria;
\item
strict equilibria are stable and attracting;
and
\item
empirical distributions of play converge to equilibrium in $2$-player zero-sum games with an interior equilibrium \citep{HS98,San10}.
\end{inparaenum}
More recently, \cite{Sor09} also showed that \eqref{eq:EW} is \emph{universally consistent}, i.e. players have no regret for following \eqref{eq:EW} instead of any other fixed strategy in a dynamically changing environment.

On the other hand, a crucial limitation in the above considerations is that players are assumed to have perfect observations of their actions' rewards.
In practical applications of game theory (e.g. in economics, finance and traffic networks), ``perfect feedback'' requirements are often too stringent, especially in games with massively many actions and/or players.
Accordingly, an important question that arises is whether \eqref{eq:EW} retains its rationality properties in the presence of noise and stochastically perturbed payoff observations.
Somewhat surprisingly, this is indeed the case (at least for some of them):
even if the payoffs in \eqref{eq:EW} are perturbed by a Brownian noise term of arbitrarily high variance, dominated strategies become extinct and strict equilibria remain stable and attracting with high probability \citep{MM09,MM10}.
Thus, even when the players' true payoffs are masked by noise and uncertainty, the reinforcement learning principle behind \eqref{eq:EW} allows players to weed out the noise and leads to similar outcomes as in the noiseless, deterministic regime. 

Motivated by the robustness of exponential learning in noisy environments, we examine here a broad class of game-theoretic learning procedures where players adjust their strategies by playing an approximate best response to the vector of their actions' cumulative payoffs \textendash\ possibly subject to random disturbances and noise.
In the case of perfect payoff observations, this scheme boils down to the reinforcement learning dynamics considered by \cite{MS16} who showed that the properties discussed above still hold in the absence of noise.
With this in mind, our main contribution is to show that feedback noise does not really matter:
even if the players' payoff observations are subject to \emph{arbitrarily high} (and possibly state-dependent or correlated) noise, dominated strategies become extinct (roughly at the same rate as in deterministic environments);
strict Nash equilibria remain (stochastically) stable and attracting;
and, in the converse direction,
if a state is (stochastically) stable or attracting with positive probability, then it is also a Nash equilibrium.
Finally, if players use a decreasing learning parameter to adjust the weight of their scoring process over time, the stochastic learning dynamics under study lead to no regret (a.s.) and their time average converges to equilibrium in $2$-player zero-sum games with an interior equilibrium.

Our analysis also highlights an important difference between ``static'' solution concepts (such as Nash equilibria and strategic dominance) and more ``dynamic'' notions (such as regret minimization).
Whereas the speed of convergence to static target states is accelerated by the use of a large, constant learning parameter (noise and disturbances notwithstanding), the rate of regret minimization is optimized by using a learning parameter that decays proportionally to $t^{-1/2}$ (and which guarantees an $\bigoh(\sqrt{t\log\log t})$ bound for the players' cumulative regret under uncertainty).
This disparity only appears in the noisy regime and is due to the fact that players need to be more conservative when facing a fluid environment that varies with time in an (a priori) unpredictable fashion.
Otherwise, if players have access to noiseless payoff observations, they can be signficantly more greedy and achieve lower regret faster by using a constant learning parameter.

\subsection{Related work}

The long-term rationality properties of exponential learning in a game-theoretic setting were first studied in conjunction with those of the \emph{replicator dynamics}, one of the most widely studied dynamical systems for population evolution under natural selection.
Indeed, a simple differentiation of \eqref{eq:EW} reveals that the evolution of the players' mixed strategies under \eqref{eq:EW} follows the differential equation:
\begin{equation}
\label{eq:RD}
\tag{RD}
\dot x_{k\alpha}
	= x_{k\alpha} \left[ \payv_{k\alpha} - \insum_{\beta} x_{k\beta} \payv_{k\beta} \right],
\end{equation}
which is simply the (multi-population) replicator dynamics of \cite{TJ78}.
In this context, \cite{Aki80}, \cite{Nac90} and \cite{SZ92} showed that dominated strategies become extinct, while it is well known that
\begin{inparaenum}
[\itshape a\upshape)]
\item
stable states of \eqref{eq:RD} are Nash equilibria;
\item
strict Nash equilibria are asymptotically stable; 
and
\item
time averages of replicator orbits converge to equilibrium in $2$-player games provided that no strategy share becomes arbitrarily small \citep{HS98,HS03}.
\end{inparaenum}

This two-way relationship between exponential learning and the replicator dynamics was noted early on by \cite{Rus99} in his study of reinforcement learning models for games \textendash\ i.e. learning how to react to a given situation so as to maximize a numerical reward \citep{SB98}.
From a game-theoretic point of view, the learning models of \cite{BS97} and \cite{ER98} are also closely related to the replicator dynamics \textendash\ and, hence, to \eqref{eq:EW} \textendash\ 
while \cite{FL98}, \cite{HS02}, and \cite{Hop02} studied a smooth variant of the well-known fictitious play algorithm where the players play a perturbed best response to their opponents' empirical frequency of play.%
\footnote{For a closely related model, see also \cite{CMS10}.}
Up to mild technical differences, the correlated version of these models can be seen as a noiseless version of our reinforcement learning model, viz. the dynamics of \cite{MS16} with the parameter choice $\temp(t) = 1/t$;
we explore this relation in more detail in Sections \ref{sec:regret} and \ref{sec:averages}.


Of course, a crucial aspect of these considerations is whether players have accurate observations of their actions' rewards or only a noisy estimate thereof:
if the former is not the case, noise and fluctuations could potentially lead to suboptimal outcomes with high probability.
In biology and evolutionary game theory (where payoffs measure the reproductive fitness of a biological species or the average payoff of populations of nonatomic players respectively), \cite{FH92} accounted for such fluctuations by introducing a stochastic variant of the replicator dynamics where evolution is perturbed by ``aggregate shocks'' that reflect the impact of weather-like effects and other perturbations.%
\footnote{\cite{KP06} also consider a Stratonovich-based model while \cite{Vla12} examines the case of random jumps incurred by catastrophic, earthquake-like events}
In this framework, \cite{Cab00}, \cite{Imh05} and \cite{HI09} showed that dominated strategies are still eliminated if the variability of the shocks across different genotypes (strategies) is not too high, while \cite{Imh05} and \cite{HI09} showed that strict Nash equilibria of a modified game are stochastically asymptotically stable under the replicator dynamics with aggregate shocks.%
\footnote{Whether the equilibria of the original game are themselves asymptotically stable, depends on the intensity of the noise on different strategies and also on the exact way that the noise enters the process \citep{MV16}.}

On the other hand, \cite{MM09,MM10} showed that the dynamics obtained by \cite{FH92} do not coincide with the stochastic replicator dynamics induced by \eqref{eq:EW} in the presence of random disturbances and measurement noise \textendash\ in contrast to the noiseless case where \eqref{eq:EW} and \eqref{eq:RD} \emph{do} coincide.
As we mentioned above, this learning variant of the stochastic replicator dynamics actually retains the rationality properties of the deterministic system \eqref{eq:EW}/\eqref{eq:RD} without any caveats on the noise:
dominated strategies become extinct and strict Nash equilibria remain stochastically asymptotically stable without any conditions on the perturbations' magnitude.
This shows that the origin of the noise is crucial in the determination of the dynamics' long-term properties:
whereas the ``aggregate shocks'' of evolutionary environments lead to rational behavior in a modified game, no such modifications are required when learning under uncertainty.

\subsection{Paper outline}

In Section \ref{sec:model}, we present our model for learning in the presence of noise and we derive the system of coupled \acp{SDE} that governs the evolution of the players' mixed strategies.
In Section \ref{sec:regret}, we show that if the players use a smoothly decreasing learning parameter, then the learning dynamics under study lead to no regret (a.s.), whatever the noise level.
In Section \ref{sec:dominated}, we investigate dominated strategies:
we show that dominated strategies become extinct (a.s.) and we derive an explicit bound for their extinction rate.
Section \ref{sec:folk} focuses on the dynamics' long-term stability and convergence properties:
we show that
\begin{inparaenum}
[\itshape a\upshape)]
\item
stochastically (Lyapunov) stable states and states that attract trajectories of play with positive probability are Nash equilibria;
and
\item
strict Nash equilibria are stochastically asymptotically stable, irrespective of the fluctuations' magnitude.
\end{inparaenum}
In Section \ref{sec:averages}, we provide an averaging principle for $2$-player games in the spirit of \cite{HSV09};
thanks to this principle, we then show that empirical distributions of play converge to Nash equilibrium in zero-sum games (again, no matter the noise level).
Finally, in Section \ref{sec:discussion}, we discuss the relaxation of some of the assumptions on the noise process \textendash\ and, in particular, the independence of the observation noise across players and strategies.

To streamline our presentation, we included several numerical illustrations in the main text (Figs. \ref{fig:portraits} and \ref{fig:averages}) and we relegated the more convoluted proofs to a series of appendices at the end.


\section{The model}
\label{sec:model}

After a few preliminaries to set notation and terminology, this section focuses on the basic properties of a broad class of reinforcement learning dynamics under noise and uncertainty.
In the noiseless, deterministic regime (Section \ref{sec:deterministic}), our model essentially boils down to the class of dynamics recently studied by \cite{MS16}.
The full stochastic framework (Section \ref{sec:stochastic}) is then obtained by positing that the agents' observations are perturbed at each moment in time by a zero-mean stochastic process (an Itô diffusion).

\subsection{Preliminaries}
\label{sec:prelims}

\paragraph{Notation}

If $\set = \{s_{\alpha}\}_{\alpha=1}^{n}$ is a finite set, the real vector space generated by $\set$ will be denoted by $\R^{\set}$ and we will write $\{\bvec_{s}\}_{s\in\set}$ for its canonical basis;
for concision, we will also use $\alpha$ to refer interchangeably to $\alpha$ or $s_{\alpha}$, writing e.g. $x_{\alpha}$ instead of $x_{s_{\alpha}}$.
The set $\simplex(\set)$ of probability measures on $\set$ will be identified with the $n$-dimensional simplex $\simplex = \{x\in \R^{\set}: \sum_{\alpha} x_{\alpha} =1 \text{ and }x_{\alpha}\geq 0\}$ and the relative interior of $\simplex$ will be written $\intsimplex$.

If $\{\set_{k}\}_{k\in\play}$ is a finite family of sets, we will use the shorthand $(\alpha_{k};\alpha_{-k})$ for the tuple $(\dotsc,\alpha_{k-1},\alpha_{k},\alpha_{k+1},\dotsc)$ and we will write $\sum_{\alpha}^{k}$ instead of $\sum_{\alpha\in\set_{k}}$.
Finally, any statement of the form $f(t) = o(g(t))$ for a stochastic process $f$ should be interpreted in the a.s. sense, i.e. ``for every $C>0$, there exists a.s. some (random) $t_{0}$ such that $\abs{f(t)} \leq C \abs{g(t)}$ for all $t\geq t_{0}$'' \textendash\ and likewise for the statement ``$f(t) = \bigoh(g(t))$''.

\paragraph{Definitions from game theory}

A \emph{finite game in normal form} is a tuple $\game \defeq \game(\play,\act,\pay)$ consisting of
\begin{inparaenum}[\itshape a\upshape)]
\item
a finite set of \emph{players} $\play = \{1,\dotsc,N\}$;
\item
a finite set $\act_{k}$ of \emph{actions} (or \emph{pure strategies}) per player $k\in\play$;
and
\item
the players' \emph{payoff functions} $\pay_{k}\from \act\to \R$,
where $\act\equiv\prod_{k}\act_{k}$ denotes the set of all joint action profiles $(\alpha_{1},\dotsc,\alpha_{N})$.
\end{inparaenum}
The set of \emph{mixed strategies} of player $k$ is denoted by $\strat_{k}\equiv\simplex(\act_{k})$ and the space $\strat\equiv\prod_{k}\strat_{k}$ of \emph{mixed strategy profiles} $x = (x_{1},\dotsc,x_{N})$ will be called the game's \emph{strategy space}.
In this mixed context, the expected payoff of player $k$ in the strategy profile $x = (x_{1},\dotsc,x_{N})\in \strat$ is
\begin{equation}
\pay_{k}(x)
	= \insum_{\alpha_{1}}^{1}\dotsi \insum_{\alpha_{N}}^{N}
	\pay_{k}(\alpha_{1},\dotsc,\alpha_{N}) \; x_{1,\alpha_{1}} \dotsm\, x_{N,\alpha_{N}},
\end{equation}
where, in a slight abuse of notation, $\pay_{k}(\alpha_{1},\dotsc,\alpha_{N})$ denotes the payoff of player $k$ in the profile $(\alpha_{1},\dotsc,\alpha_{N})\in\act$.
Accordingly, the payoff corresponding to $\alpha\in\act_{k}$ in the mixed strategy profile $x\in\strat$ is
\begin{flalign}
\payv_{k\alpha}(x)
	\equiv \pay_{k}(\alpha;x_{-k})
	&= \insum_{\alpha_{1}}^{1} \dotsi \insum_{\alpha_{N}}^{N}
	\pay_{k}(\alpha_{1},\dotsc,\alpha_{N})
	\; x_{1,\alpha_{1}} \dotsm\,\delta_{\alpha_{k},\alpha} \dotsm\, x_{N,\alpha_{N}},
\end{flalign}
leading to the more concise expression
\begin{equation}
\pay_{k}(x)
	= \insum_{\alpha}^{k} x_{k\alpha} \payv_{k\alpha}(x)
	= \braket{\payv_{k}(x)}{x_{k}}
\end{equation}
where $\payv_{k}(x) = (\payv_{k\alpha}(x))_{\alpha\in\act_{k}}$ denotes the \emph{payoff vector} of player $k$ at $x\in\strat$.

\subsection{The deterministic model}
\label{sec:deterministic}

Our learning model will be based on the following simple idea:
the game's players ``score'' their actions by keeping track of their cumulative payoffs and then, at each moment $t\geq0$, they play an ``approximate'' best response to this vector of performance scores
(thus reinforcing the probability of playing an action with a higher overall payoff).
More precisely, following \cite{MS16}, this process can be described by the dynamics
\begin{equation}
\label{eq:RL}
\tag{RL}
\begin{aligned}
y_{k}(t)
	&= \int_{0}^{t} \payv_{k}(x(s)) \dd s,
	\\
x_{k}(t)
	&= \choice_{k}(\temp_{k}(t) y_{k}(t)),
\end{aligned}
\end{equation}
where:
\begin{enumerate}
\item
the \emph{score vector} $y_{k}(t) = (y_{k\alpha}(t))_{\alpha\in\act_{k}}$ of player $k$ ranks each strategy $\alpha\in\act_{k}$ based on its cumulative payoff $y_{k\alpha}(t) = \int_{0}^{t} \payv_{k\alpha}(x(s)) \dd s$ up to time $t$.
\item
$\choice_{k}(y_{k})$
is a \emph{choice map} which reinforces the strategies with the highest scores (see below for a rigorous definition).
\item
$\temp_{k}(t)>0$ is a \emph{learning parameter} which can be tuned freely by each player.
\end{enumerate}

Clearly, a natural choice for the ``scores-to-strategies'' map $\choice_{k}$ would be to take the $\argmax$ correspondence $y_{k} \mapsto\argmax_{x_{k} \in\strat_{k}} \braket{y_{k}}{x_{k}}$, i.e. to greedily assign all weight to the strategy (or strategies) with the highest score.
However, since the $\argmax$ operator is multi-valued (so each player would also need to employ some tie-breaking rule to resolve ambiguities), we will focus on \emph{single-valued} choice maps of the general form
\begin{equation}
\label{eq:choice}
\choice_{k}(y_{k})
	= \argmax_{x_{k}\in\strat_{k}}\{\braket{y_{k}}{x_{k}} - h_{k}(x_{k})\},
\end{equation}
where the \emph{penalty function} $h_{k}\from\strat_{k}\to\R$ satisfies the following properties:
\begin{enumerate}
[\indent\,\itshape a\upshape)]
\setlength{\itemsep}{0pt}
\item
$h_{k}$ is continuous on $\strat_{k}$.
\item
$h_{k}$ is smooth on the relative interior of every face of $\strat_{k}$.%
\item
$h_{k}$ is \emph{strongly convex} on $\strat_{k}$, i.e. there exists some $K>0$ such that
\begin{equation}
\label{eq:strong}
h_{k}(tx_{k} + (1-t) x_{k}')
	\leq th_{k}(x_{k}) + (1-t) h_{k}(x_{k}') - \tfrac{1}{2} K t (1-t) \smallnorm{x_{k}' - x_{k}}^{2},
\end{equation}
for all $x_{k},x_{k}' \in \strat_{k}$ and for all $t\in[0,1]$.
\end{enumerate}

This ``softening'' of the $\argmax$ operator has a long history in game theory and optimization, and the induced map $\choice_{k}$ is intimately related to the notion of \emph{softmax} or \emph{perturbed/re\-gu\-la\-rized best response maps};
for an in-depth discussion, we refer the reader to \cite{HS02} and \cite{MS16}.
For our immediate purposes, the key observation is that \eqref{eq:choice} is a strictly concave problem, so it admits a unique solution for every input vector $y_{k}$.
Therefore, when the penalty term $h_{k}$ is small relative to $y_{k}$, $\choice_{k}(y_{k})$ can be seen as a single-valued approximation of the standard best response correspondence $y_{k} \mapsto \argmax_{x_{k}\in\strat_{k}} \braket{y_{k}}{x_{k}}$.

Regarding the learning parameter $\temp_{k}$, its role in \eqref{eq:RL} is to temper the growth of the cumulative payoff vector $y_{k}(t)$ so as to allow the player to better explore his strategies (instead of prematurely reinforcing one or another).
In other words, $\temp_{k}$ can be interpreted as an extrinsic weight that the player assigns to cumulative observations of past payoffs.%
\footnote{The role of the learning parameter $\temp(t)$ can also be linked to the vanishing step-size rules that are used in the theory of stochastic approximation \textendash\ see e.g. \cite{Ben99}, \cite{LPT04}, \cite{OR14} and references therein.
The difference between the two is that, in the theory of stochastic approximation, a variable step-size means that new information enters the algorithm with decreasing weight;
on the other hand, in our context, all information is weighted evenly, but the \emph{aggregate} information is weighted by $\temp(t)$ to avoid extreme behaviors.}
Accordingly, given that $y(t)$ grows (at most) as $\bigoh(t)$, we will assume throughout that:
\begin{assumption}
\label{asm:temp}
$\temp_{k}(t)$ is $C^{1}$-smooth, nonincreasing and $\lim_{t\to\infty} t\temp_{k}(t) = +\infty$.
\end{assumption}

The decay rate of $\temp(t)$ will play a crucial role when the players' payoff observations are subject to stochastic perturbations, an issue that we will explore in detail in later sections.
For now, we turn to two representative examples of  \eqref{eq:RL}:

\begin{example}
\label{ex:logit}
The prototype penalty function on the simplex is the Gibbs (negative) entropy $h(x) = \insum_{\alpha} x_{\alpha} \log x_{\alpha}$ which, after a standard calculation, yields the so-called \emph{logit map}:
\begin{equation}
\label{eq:logit}
\gibbs_{\alpha}(y)
	= \frac{\exp(y_{\alpha})}{\insum_{\beta} \exp(y_{\beta})}.
\end{equation}
An easy differentiation then yields
\begin{equation}
\tag{\ref*{eq:RD}}
\dot x_{k\alpha}
	= \frac{e^{y_{k\alpha}} \dot y_{k\alpha}}{\insum_{\beta}^{k} e^{y_{k\beta}}}
	- \frac{e^{y_{k\alpha}} \insum_{\beta}^{k} e^{y_{k\beta}} \dot y_{k\beta}}{\left( \insum_{\beta}^{k} e^{y_{k\beta}} \right)^{2}}
	= x_{k\alpha} \left[ \payv_{k\alpha}(x) - \insum_{\beta}^{k} x_{k\beta}\,\payv_{k\beta}(x) \right],
\end{equation}
which is simply the (multi-population) replicator equation of \cite{TJ78} for population evolution under natural selection.
For a more thorough treatment of the links between logit choice and the replicator dynamics, see \cite{HSV09}, \cite{MM10}, and references therein.
\end{example}

\begin{example}
\label{ex:proj}
As another example, consider the penalty function $h(x) = \frac{1}{2} \insum_{\alpha} x_{\alpha}^{2}$.
This quadratic penalty function leads to the \emph{projected choice map}
\begin{equation}
\label{eq:Eucl}
\eucl(y)
	= \argmin\nolimits_{x\in\simplex} \big\{ \braket{y}{x} - \tfrac{1}{2} \norm{x}^{2} \big\}
	= \argmin\nolimits_{x\in\simplex} \norm{y-x}^{2},
\end{equation}
and, as was shown by \cite{MS16}, the induced trajectories $x(t) = \eucl(y(t))$ of \eqref{eq:RL} satisfy the so-called \emph{projection dynamics}
\begin{equation}
\label{eq:PD}
\tag{PD}
\dot x_{k\alpha}
	=\begin{cases}
		\payv_{k\alpha}(x) - \abs{\supp(x_{k})}^{-1} \insum_{\beta\in\supp(x_{k})} \payv_{k\beta}(x)
		&\quad
		\text{if $x_{k\alpha}>0$,}
		\\
		0
		&\quad
		\text{otherwise,}
	\end{cases}
\end{equation}
on an open dense set of times (in particular, except when the support of $x(t)$ changes).
The dynamics \eqref{eq:PD} were introduced in game theory by \cite{Fri91} as a geometric model of the evolution of play in population games;
for a closely related model (but with different long-term properties), see \cite{NZ97} and \cite{LS08}.
\end{example}

\paragraph{Related models}


The reinforcement mechanism of \eqref{eq:RL} is seen quite clearly in the work of  \cite{Vov90} and \cite{LW94} on multi-armed bandits.
Therein, the agent is faced with a repeated decision process (e.g. choosing a slot machine in the eponymous problem) and, at each stage, he selects an action with probability exponentially proportional to an estimate of said action's cumulative payoff up to that time.
In this way, the mean dynamics of the agent's learning process boil down to the exponential learning scheme \eqref{eq:EW} of Example \ref{ex:logit} \textendash\ itself a special case of \eqref{eq:RL}.


\cite{LC05} and \cite{CGM15} also considered a reinforcement learning process where players discount past observations by a constant multiplicative factor and then play a perturbed best response to the resulting cumulative payoff vector \textendash\ or estimate thereof.
The mean dynamics that describe this process in continuous time are then given by \eqref{eq:RL} with constant $\temp$ and an additional adjustment term that accounts for the exponential discounting of past observations.
For a more detailed survey of the surrounding literature, we refer the reader to \cite{FL98} and \cite{MS16}.

\subsection{Learning in the presence of noise}
\label{sec:stochastic}

A key assumption underlying the reinforcement learning scheme \eqref{eq:RL} is that the players' payoff observations are impervious to any sort of exogenous random noise.
However, this assumption is rarely met in practical applications of game-theoretic learning:
for instance, in telecommunication networks and traffic engineering, signal strength and latency measurements are constantly subject to stochastic fluctuations which introduce noise to the input of any learning algorithm \citep{KMT98,KKLW09}.
Thus, to account for the lack of accurate payoff information in settings where uncertainty is an issue, we will consider the \emph{stochastically perturbed reinforcement learning} process:
\begin{equation}
\label{eq:SRL}
\tag{SRL}
\begin{aligned}
dY_{k\alpha}
	&= \payv_{k\alpha}(X) \dd t
	+ \sigma_{k\alpha}(X) \dd W_{k\alpha},
	\\
X_{k}
	&= \choice_{k}(\temp_{k} Y_{k}),
\end{aligned}
\end{equation}
where $W_{k\alpha}$ is a family of standard Wiener processes (assumed independent across players and strategies) and the diffusion coefficients $\sigma_{k\alpha}\from\strat\to\R$ (assumed Lipschitz) measure the strength of the noise process.

Intuitively, the random component of \eqref{eq:SRL} means that the players' observed payoffs are only accurate up to an error term of the order of $\sigma_{k\alpha}$, possibly depending on the players' mixed strategies but otherwise uncorrelated over players and strategies.%
\footnote{An alternative source of stochasticity could be any inherent randomness in the players' payoffs \textendash\ for instance, if there is a random component to the players' payoffs.
A detailed analysis of this case would require a careful reformulation of the underlying game which, for simplicity, we do not attempt here;
for a related treatment in an evolutionary context, see \cite{MV16}.}
As such, \eqref{eq:SRL} should be viewed as a special instance of the more general case where observation errors exhibit correlations between different actions:
for instance, if each player's action corresponds to a choice of route in a congestion game, any two routes that overlap will exhibit such correlations.
The impact of including such correlations in our model is discussed at length in Section \ref{sec:discussion};
however, for simplicity, our analysis will be stated in the uncorrelated case.

Regarding the existence and uniqueness of solutions to \eqref{eq:SRL}, Proposition \ref{prop:Fenchel} in Appendix \ref{app:auxiliary} shows that the players' choice maps $\choice_{k}$ are Lipschitz continuous, so \eqref{eq:SRL} admits a unique (strong) solution $Y(t)$ for every initial condition $Y(0)$.
Standard arguments can then be used to show that these solutions exist for all time (a.s.),
so the players' mixed strategy profile $X(t) = \choice(\temp(t) Y(t))$ is also fully determined for all $t\geq0$.
However, since this is a somewhat indirect description of the evolution of $X(t)$,
our first task will be to derive the governing dynamics of $X(t)$ in the form of a \acl{SDE} stated directly on $\strat$.

For simplicity, we only present here the special case where each player's penalty function is of the \emph{separable} form:
\begin{equation}
\label{eq:decomposable}
h_{k}(x_{k})
	= \insum_{\alpha}^{k} \theta_{k}(x_{k\alpha})
\end{equation}
for some strongly convex \emph{kernel function} $\theta_{k}\in C^{0}[0,1] \cap C^{3}(0,1]$.
We then get:

\begin{proposition}
\label{prop:evolution}
Let $X(t)$ be an orbit of \eqref{eq:SRL} in $\strat$ and let $\interval$ be a random open interval such that $\supp(X(t))$ remains constant over $\interval$.
Then, for all $t\in\interval$, $X(t)$ satisfies the \acl{SDE}
\begin{subequations}
\label{eq:evolution}
\begin{flalign}
dX_{k\alpha}
	&\label{eq:evolution-core}
	= \frac{\temp_{k}}{\theta_{k\alpha}''}
	\left[
	\payv_{k\alpha} - \Theta_{k}'' \insum_{\beta} \payv_{k\beta} \big/\theta_{k\beta}''
	\right] dt
	\\
	&\label{eq:evolution-noise}
	+ \frac{\temp_{k}}{\theta_{k\alpha}''}
	\left[
	\sigma_{k\alpha} \dd W_{k\alpha} - \Theta_{k}'' \insum_{\beta} \sigma_{k\beta} \big/ \theta_{k\beta}'' \dd W_{k\beta}
	\right]
	\\
	&\label{eq:evolution-temp}
	+\frac{\dot \temp_{k}}{\temp_{k}} \frac{1}{\theta_{k\alpha}''}
	\left[
	\theta_{k\alpha}' - \Theta_{k}'' \insum_{\beta} \theta_{k\beta}' \big/ \theta_{k\beta}''
	\right] dt
	\\
	&\label{eq:evolution-Ito}
	-\frac{1}{2} \frac{1}{\theta_{k\alpha}''}
	\left[ \theta_{k\alpha}''' U_{k\alpha}^2 - \Theta_{k}'' \insum_{\beta} \theta_{k\beta}'''/\theta_{k\beta}''\, U_{k\beta}^2
	\right] dt,
\end{flalign}
\end{subequations}
where all summations are taken over $\beta\in\supp(X_{k})$ and:
\begin{enumerate}
[\indent\itshape a\upshape)]
\setlength{\itemsep}{1pt}
\item
$\theta_{k\alpha}' = \theta_{k}'(X_{k\alpha})$, $\theta_{k\alpha}'' = \theta_{k}''(X_{k\alpha})$, $\theta_{k\alpha}''' = \theta_{k}'''(X_{k\alpha})$,
\item
$\Theta_{k}'' = \left(\insum_{\beta} 1/\theta_{k\beta}''\right)^{-1}$,
\item
\(
\dis
U_{k\alpha}^{2}
	= \left(\frac{\temp_{k}}{\theta_{k\alpha}''}\right)^{2}
	\left[
		\sigma_{k\alpha}^{2} \left(1 - \Theta_{k}''\big/\theta_{k\alpha}''\right)^{2}
		+ \insum_{\beta\neq\alpha}
		\left( \Theta_{k}''\big/\theta_{k\beta}'' \right )^{2} \sigma_{k\beta}^{2}
	\right]
\).
\end{enumerate}
\vspace{1ex}
In particular, if $\lim_{z\to0^{+}} \theta_{k}'(z) = -\infty$ for all $k\in\play$, $X(t)$ is an ordinary \textup(strong\textup) solution of \eqref{eq:evolution};
otherwise, $X(t)$ satisfies \eqref{eq:evolution} on a random open dense subset of $[0,\infty)$.
\end{proposition}

The proof of Proposition \ref{prop:evolution} is a simple but interesting application of It\^o's formula;
to streamline our presentation, we relegate it to Appendix~\ref{app:main} and we focus here on some remarks and examples (see also Fig.~\ref{fig:portraits} for some sample trajectories of the process):


\begin{remark}
Even though the dynamics \eqref{eq:evolution} appear quite convoluted, each of the constituent terms \eqref{eq:evolution-core}\textendash\eqref{eq:evolution-Ito} admits a relatively simple interpretation:
\begin{enumerate}
[\indent\itshape a\upshape)]
\setlength{\itemsep}{0pt}
\item
The term \eqref{eq:evolution-core} drives the process in the case $\sigma=0$, $\temp=\text{constant}$.
These are the deterministic dynamics studied by \cite{MS16} and correspond to learning in settings with no uncertainty.

\item
The diffusion term \eqref{eq:evolution-noise} reflects the direct impact of the noise on \eqref{eq:SRL}.

\item
The term \eqref{eq:evolution-temp} is due to the variability of the players' learning rate $\temp$ so its impact on \eqref{eq:evolution} vanishes if $\temp(t)\to0$ sufficiently fast.
On the other hand, this term obviously persists even if there is no noise in the players' learning process.

\item
Finally, the term \eqref{eq:evolution-Ito} is the Itô correction induced on $X_{k}$ through \eqref{eq:SRL} due to the non-anticipative nature of the Itô integral.%
\footnote{If \eqref{eq:SRL} had been formulated as a Stratonovich equation, \eqref{eq:evolution-Ito} would vanish;
however, the future-anticipating nature of the Stratonovich integral \citep{vK81} is not well-suited for our purposes.}
Importantly, this term depends on $\temp$ but not $\dot\temp$, so it does not vanish for constant $\temp$.
\end{enumerate}
\end{remark}

\begin{example}
As we saw in Example \ref{ex:logit}, the replicator dynamics \eqref{eq:RD} correspond to the entropic kernel $\theta(x) = x \log x$. In this case, \eqref{eq:evolution} leads to the following stochastic variant of the replicator dynamics:
\begin{equation}
\label{eq:SRD}
\tag{SRD}
\begin{aligned}
dX_{k\alpha}
	&= \temp_{k} X_{k\alpha}
	\left[
	\payv_{k\alpha} - \insum_{\beta}^{k} X_{k\beta}\,\payv_{k\beta}
	\right] dt
	\\
	&+ \temp_{k} X_{k\alpha}
	\left[
	\sigma_{k\alpha} \dd W_{k\alpha} - \insum_{\beta}^{k} \sigma_{k\beta} X_{k\beta} \dd W_{k\beta}
	\right]
	\\
	&+\frac{\dot \temp_{k}}{\temp_{k}} X_{k\alpha}
	\left[
	\log X_{k\alpha} - \insum_{\beta}^{k} X_{k\beta} \log X_{k\beta}
	\right] dt
	\\
	&+\frac{\temp_{k}^{2}}{2} X_{k\alpha}
	\left[
	\sigma_{k\alpha}^{2} (1 - 2 X_{k\alpha}) - \insum_{\beta}^{k} \sigma_{k\beta}^{2} X_{k\beta}\,(1 - 2 X_{k\beta})
	\right] dt.
\end{aligned}
\end{equation}
For constant $\temp$, \eqref{eq:SRD} is simply the stochastic replicator dynamics of exponential learning studied by \cite{MM10}.
As such, \eqref{eq:SRD} should be contrasted to the evolutionary \emph{replicator dynamics with aggregate shocks} of \cite{FH92}:
\begin{equation}
\label{eq:ASRD}
\tag{ASRD}
\begin{aligned}
dX_{k\alpha}
	&= X_{k\alpha} \left[ \payv_{k\alpha} - \insum_{\beta}^{k} X_{k\beta} \payv_{k\beta} \right] dt
	\\
	&+ X_{k\alpha} \left[
	\sigma_{k\alpha} dW_{k\alpha} - \insum_{\beta}^{k} \sigma_{k\beta} X_{k\beta} \dd W_{k\beta}
	\right]
	\\
	&- X_{k\alpha} \left[
	\sigma_{k\alpha}^{2} X_{k\alpha} - \insum_{\beta}^{k} \sigma_{k\beta}^{2} X_{k\beta}^{2}
	\right] dt,
\end{aligned}
\end{equation}
where $X_{k\alpha}$ denotes the population share of the $\alpha$-th genotype of species $k$ in a multi-species environment, $\payv_{k\alpha}$ represents its reproductive fitness, and the noise coefficients $\sigma_{k\alpha}$ measure the impact of random weather-like effects on population evolution.%
\footnote{For a comprehensive account of the literature surrounding \eqref{eq:ASRD}, see \cite{HI09}, \cite{MV16}, and references therein.
In addition to the works mentioned above, \cite{KP06} study a Stratonovich-based formulation of \eqref{eq:ASRD} while \cite{Vla12} also considers random jumps induced by discontinuous, earthquake-like events.}

Besides the absence of the learning rate $\temp$, the fundamental difference between \eqref{eq:SRD} and \eqref{eq:ASRD} is in their Itô correction:
as we shall see in what follows, this term leads to a drastically different long-term behavior and highlights an important contrast between learning and evolution in the presence of noise.
\end{example}

\begin{example}
In the case of the projected reinforcement learning scheme \eqref{eq:PD}, substituting $\theta(x) = x^{2}/2$ in \eqref{eq:evolution} yields the \emph{stochastic projection dynamics}:
\begin{equation}
\label{eq:SPD}
\tag{SPD}
\begin{aligned}
dX_{k\alpha}
	&= \left[
	\payv_{k\alpha} - \abs{\supp(X_{k})}^{-1} \insum_{\beta\in\supp(X_{k})} \payv_{k\beta}
	\right] dt
	\\
	&+ \left[
	\sigma_{k\alpha} \dd W_{k\alpha} - \abs{\supp(X_{k})}^{-1} \insum_{\beta\in\supp(X_{k})} \sigma_{k\beta} \dd W_{k\beta}
	\right]
	\\
	&+\frac{\dot \temp_{k}}{\temp_{k}}
	\left[
	X_{k\alpha} - \abs{\supp(X_{k})}^{-1}
	\right] dt.
\end{aligned}
\end{equation}
There are two important qualitative differences between \eqref{eq:SRD} and \eqref{eq:SPD}:
first, \eqref{eq:SRD} holds for all $t\geq0$ whereas \eqref{eq:SPD} describes the solution orbits of \eqref{eq:SRL} only on intervals over which the support of $X(t)$ remains constant.
Second, the projection mapping $\eucl$ of \eqref{eq:Eucl} is piecewise linear, so there is no Itô correction in \eqref{eq:SPD};
accordingly, the distinction between Itô and Stratonovich perturbations becomes void in the context of \eqref{eq:SPD}.
\end{example}

\begin{figure}[t]
\subfigure{
\label{fig:portraits-XL}
\includegraphics[width=.48\textwidth]{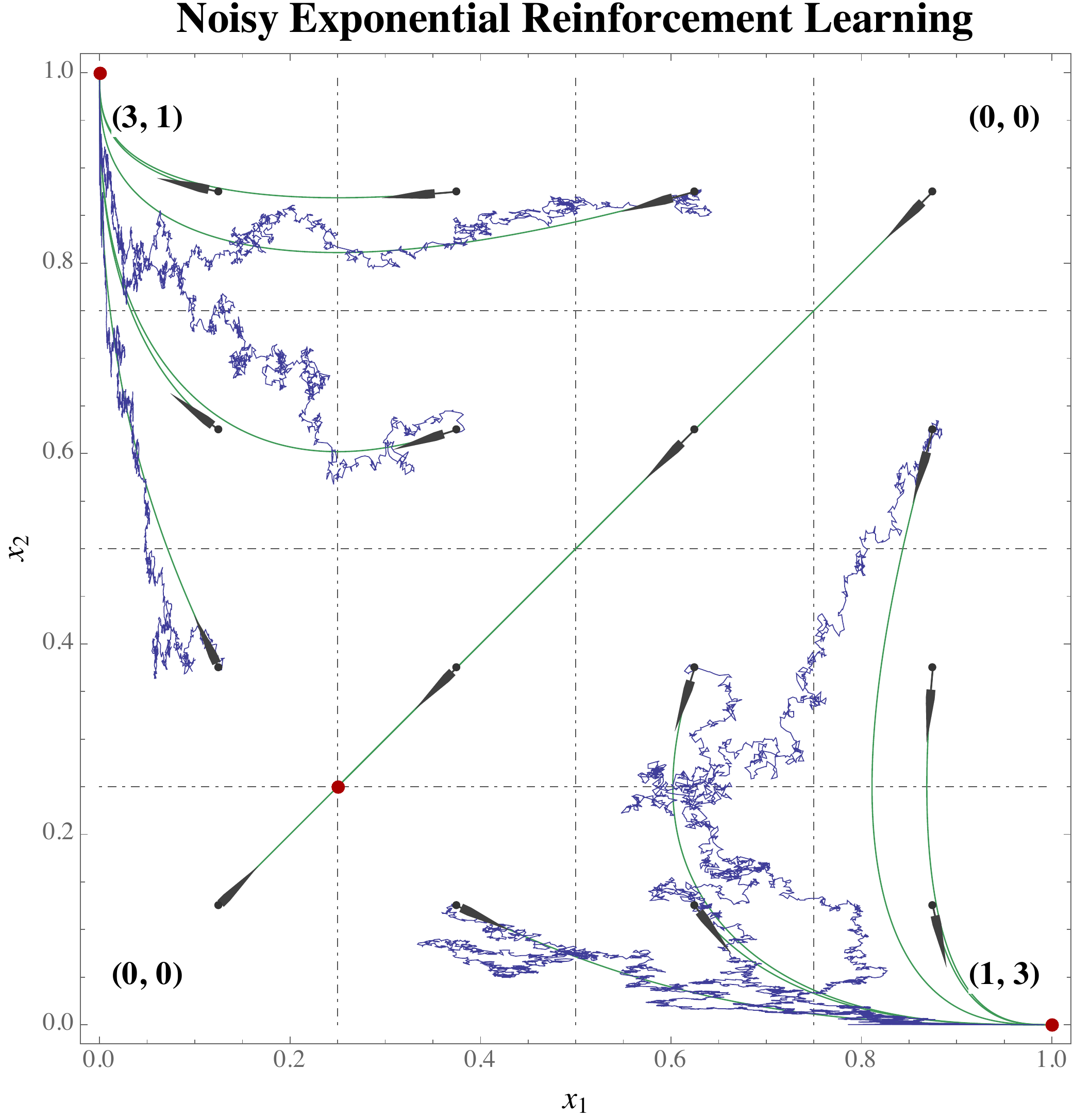}}
\hfill
\subfigure{
\label{fig:portraits-PL}
\includegraphics[width=.48\textwidth]{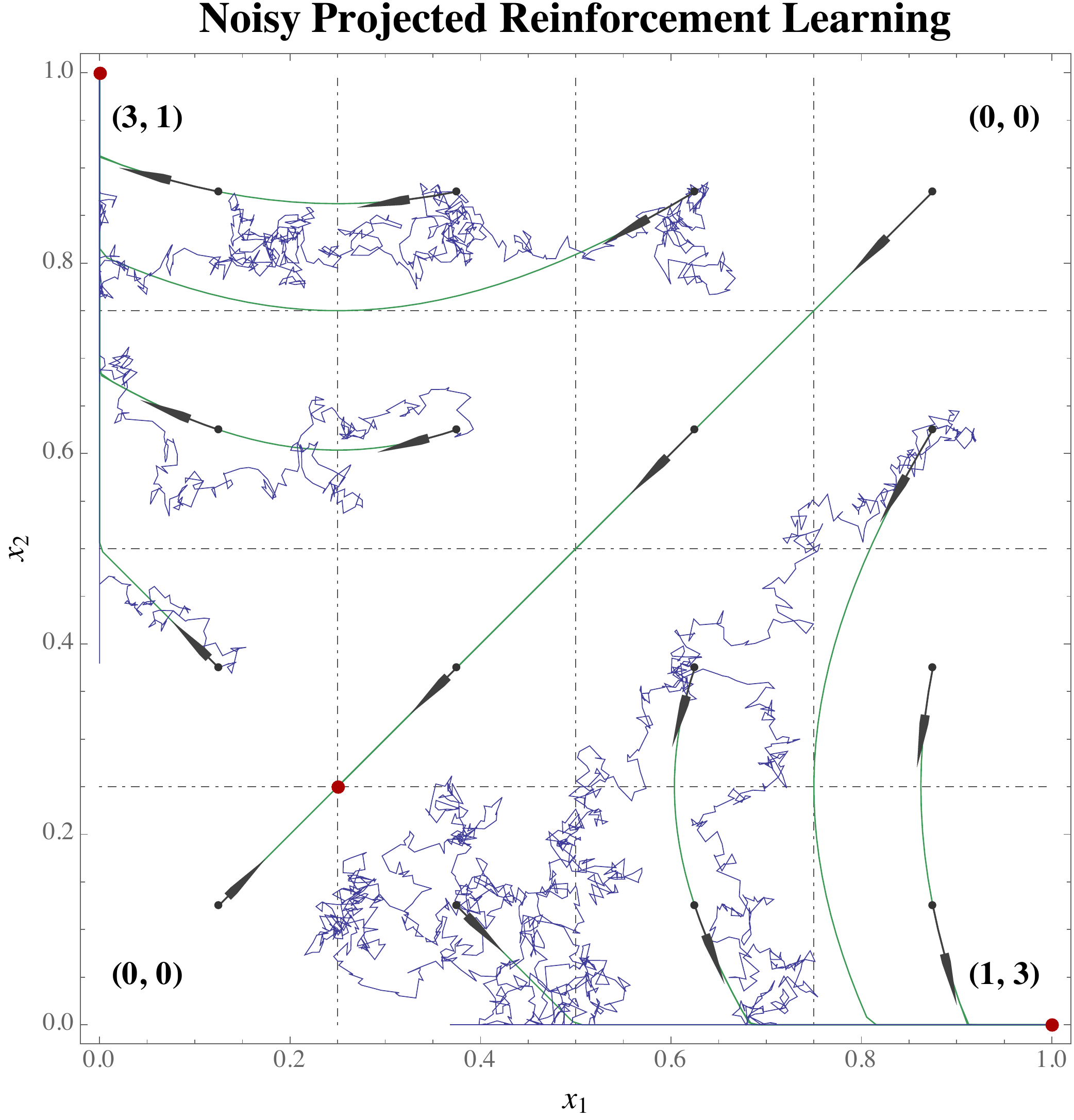}}
\quad

\caption{\footnotesize
Evolution of play under \eqref{eq:SRL} with logit \textup(left\textup) and choice maps \textup(right\textup) for a $2\times2$ congestion game
\textup(Nash equilibria are depicted in red; for the game's payoffs, see the vertex labels\textup).
For comparison purposes, we took $\sigma_{k\alpha}=1$ for all $\alpha\in\act_{k}$, $k=1,2$, and we used the same Wiener process realization in both cases:
the orbits of projected reinforcement learning attain the boundary of $\strat$ and converge to Nash equilibrium much faster than in the case of exponential learning.}
\label{fig:portraits}
\end{figure}


\section{Regret minimization}
\label{sec:regret}

We begin our rationality analysis in the case where there is a single player whose payoffs are determined by the state of his environment \textendash\ which, in turn, may evolve \emph{arbitrarily} over time (including adversarially if the player is playing against an opponent).

More precisely, consider a decision process where, at each $t\geq0$, the player chooses an action from a finite set $\act$ according to some mixed strategy $x(t)\in\strat\equiv\simplex(\act)$ and his expected payoff is determined by the (a priori unknown) payoff vector $\payv(t) = (\payv_{\alpha}(t))_{\alpha\in\act}$ of stage $t$.
In this context, the performance of a dynamic strategy $x(t)$ can be measured by comparing the player's (expected) cumulative payoff to the payoff that he could have obtained if the state of nature were known in advance and the player had best-responded to it;
specifically, the player's \emph{cumulative regret} at time $t$ is defined as
\begin{equation}
\label{eq:regret}
\reg(t)
	= \max_{\alpha\in\act} \int_{0}^{t} \payv_{\alpha}(s) \dd s
	- \int_{0}^{t} \braket{\payv(s)}{x(s)} \dd s,
\end{equation}
and we say that a strategy $x(t)$ is \emph{consistent} if it leads to \emph{no \textup(average\textup) regret}, i.e.
\begin{equation}
\label{eq:consistent}
\limsup_{t\to\infty} \frac{1}{t} \reg(t)
	\leq 0
	\quad
	\text{(a.s.)},
\end{equation}
or, equivalently:
\begin{equation}
\reg(t)
	= o(t)
	\quad
	\text{as $t\to\infty$.}
\end{equation}

The notion of consistency presented above is commonly referred to as \emph{external} or \emph{universal} consistency and was originally introduced by \cite{Han57} in a discrete-time context (for a detailed overview, see \cite{FL98}, \cite{SS11} and references therein).
There, the agent is assumed to receive a payoff vector $\payv_{n}\in\R^{\act}$ at each stage $n\in\N$ and, observing the past realizations $\payv_{1},\dotsc,\payv_{n-1}$, he chooses an action $\alpha_{n}\in\act$ according to some probability distribution $x_{n}\in\simplex(\act)$.
The strategy $x_{n}$ is then said to be \emph{externally consistent} if, on average, it earns more than any action (or expert suggestion) $\alpha\in\act$;
in other words, the no-regret criterion \eqref{eq:consistent} is simply the continuous-time analogue of the standard definition of external consistency.

With this in mind, the main question that we seek to address here is whether the perturbed reinforcement learning process \eqref{eq:SRL} is consistent.
To make this precise, we will focus on the \emph{unilateral} process
\begin{equation}
\label{eq:USRL}
\tag{\ref*{eq:SRL}-U}
\begin{aligned}
dY_{\alpha}(t)
	&= \payv_{\alpha}(t) \dd t
	+ \sigma_{\alpha}(t) \dd W_{\alpha}(t),
	\\
X(t)
	&= \choice(\temp(t) Y(t)),
\end{aligned}
\end{equation}
where $\payv(t)$ is a locally integrable stream of payoff vectors,
$W(t)$ is a Wiener process in $\R^{\act}$
and the noise coefficients $\sigma_{\alpha}$ (assumed continuous and bounded) represent the error in the player's payoff observations.
In the deterministic case ($\sigma=0$),
\cite{Sor09} proved that the unilateral variant of \eqref{eq:EW} is consistent, a result which was recently extended by \cite{KM14} to the more general setting of \eqref{eq:RL}.
Below, we show that this learning process remains consistent
even when the player's payoff observations are subject to arbitrarily high observation noise:

\begin{theorem}
\label{thm:consistent}
Assume that \eqref{eq:USRL} is run with learning parameter $\temp(t)$ satisfying $\lim_{t\to\infty} \temp(t) = 0$ \textup(in addition to Assumption \ref{asm:temp}\textup) and initial bias $Y(0) = 0$.
Then, \eqref{eq:USRL} is consistent and it enjoys the cumulative regret bound:
\begin{equation}
\label{eq:reg-bound}
\reg(t)
	\leq \frac{\depth}{\temp(t)}
	+ \sigma_{\max}^{2} \frac{\abs{\act}}{2K} \int_{0}^{t} \temp(s) \dd s
	+ \bigoh(\sigma_{\max}\sqrt{t \log\log t}),
\end{equation}
where $\depth = \max\{h(x) - h(x') : x, x' \in\strat\}$,
$\sigma_{\max} = \sup_{t} \max_{\beta} \sigma_{\beta}(t)$
and
$K$ is the strong convexity constant of the player's penalty function $h$.
\end{theorem}

\begin{remark*}
The assumption $Y(0) = 0$ is only made for simplicity:
the policy \eqref{eq:USRL} is consistent for any initial score vector $Y(0)$ but the corresponding regret guarantee is a bit more cumbersome to write down.
\end{remark*}

The basic idea of our proof will be to compare the cumulative payoff of the policy \eqref{eq:USRL} up to time $t$ to that of an arbitrary test strategy $p\in\strat$.
To that end, we will examine how far the induced trajectory of play $X(t) = \choice(\temp(t) Y(t))$ can stray from $p$;
however, since $X(t)$ is defined via the cumulative payoff process $Y(t)$, we will carry out this comparison between $p$ and $Y(t)$ \emph{directly} using the so-called \emph{Fenchel coupling}:
\begin{equation}
\label{eq:Fenchel}
\fench(x,y)
	= h(x) + h^{\ast}(y) - \braket{y}{x},
\end{equation}
where
\begin{equation}
\label{eq:conjugate-k}
h^{\ast}(y)
	= \max_{x\in\strat} \{ \braket{y}{x} - h(x)\}
\end{equation}
denotes the \emph{convex conjugate} of $h$ \citep{Roc70}.%
\footnote{The terminology ``Fenchel coupling'' is due to \cite{MS16} and reflects the fact that \eqref{eq:Fenchel} collects all the terms of Fenchel's inequality.}
By Fenchel's inequality \citep{Roc70}, $\fench(x,y)$ is non-negative in both arguments, so it provides a ``congruity'' measure between $x$ and $y$.
With this in mind, our proof strategy will be to express the player's regret with respect to a test strategy $p\in\strat$ in terms of $\fench(p,Y(t))$ and then show that the latter grows sublinearly in $t$.

The details of the proof are presented in Appendix \ref{app:main} and rely on Itô's lemma and the law of the iterated logarithm (which is used to control the impact of the noise on the agent's learning process).
For now, we focus on some aspects of Theorem \ref{thm:consistent} and the regret bound \eqref{eq:reg-bound}:

\begin{remark}[\emph{The role of $\temp$}]
It is important to note that the second term of \eqref{eq:reg-bound} becomes linear in $t$ when the player's learning parameter $\temp$ is constant, explaining in this way the requirement that $\temp(t)\to0$ as $t\to\infty$.
On the other hand, this requirement can be dropped in the noiseless case:
when $\sigma=0$, the guarantee \eqref{eq:reg-bound} reduces to $\depth/\temp(t)$, so \eqref{eq:USRL} remains consistent even for constant $\temp$ \textendash\ a fact which was first observed by \cite{Sor09} in the context of \eqref{eq:EW} and \cite{KM14} for \eqref{eq:RL}.
In the presence of noise, we conjecture that \eqref{eq:USRL} may lead to positive regret with positive probability for constant $\temp$ but we have not been able to prove this.
\end{remark}

The form of the bound \eqref{eq:reg-bound} also highlights a trade-off between more aggressive learning rates (slowly decaying $\temp$) and the noise affecting the player's payoff observations.
Specifically, the first (deterministic) term of \eqref{eq:reg-bound} is decreasing in $\temp$ while the second one (which is due to all the noise) is increasing in $\temp$;%
\footnote{The last term of \eqref{eq:reg-bound} is also due to the noise, but it does not otherwise depend on $\temp(t)$.}
as a result, in the absence of noise ($\sigma=0$), it is better to use a large, constant $\temp$ rather than letting $\temp(t)\to0$, but this might lead to disastrous results under uncertainty.
\cite{SS11} draws a similar conclusion for the discrete-time analogue of \eqref{eq:RL} known as \ac{OMD};
in this way, discretization and random payoff disturbances seem to have comparably adverse effects on the agent's learning process.


The above considerations can be illustrated more explicitly by choosing $\temp(t) = t^{-\gamma}$ for some $\gamma\in(0,1)$;
in this case, Theorem \ref{thm:consistent} yields the regret bounds:

\begin{corollary}
\label{cor:reg-rate}
Assume that \eqref{eq:USRL} is run with learning rate $\temp(t) \sim t^{-\gamma}$ for some $\gamma\in(0,1)$.
Then:
\begin{equation}
\label{eq:reg-gamma}
\reg(t)
	= \begin{cases}
	\bigoh\left(t^{1-\gamma}\right)
		&\quad
		\text{if $0<\gamma<\frac{1}{2}$},
		\\
	\bigoh\left(\sqrt{t \log \log t}\right)
		&\quad
		\text{if $\gamma=\frac{1}{2}$},
		\\
	\bigoh\left(t^{\gamma}\right)
		&\quad
		\text{if $\frac{1}{2}<\gamma<1$}.
	\end{cases}
\end{equation}
\end{corollary}

\begin{proof}
Simply note that $t^{\max\{\gamma,1-\gamma\}}$ is the dominant term in \eqref{eq:reg-bound} for all $\gamma\neq 1/2$.
Otherwise, for $\gamma=1/2$, the first two terms of \eqref{eq:reg-bound} are both $\bigoh(\sqrt{t})$ and are dominated by the third.
\end{proof}

\begin{remark}[\emph{Links with vanishingly smooth fictitious play}]
We close this section by discussing the links of \eqref{eq:USRL} with the vanishingly smooth fictitious play process that was recently examined by \cite{BF13} in a discrete-time setting.
Specifically, by interpreting $t^{-1} \int_{0}^{t} \payv(s) \dd s$ as the payoff that a player obtains in a $2$-player game against his opponent's empirical frequency of play (cf. Section \ref{sec:averages}), the strategy
\begin{equation}
x(t)
	= \choice\left(\temp(t) \int_{0}^{t} \payv(s) \dd s\right)
	= \choice\left(t\temp(t) \cdot t^{-1}\int_{0}^{t} \payv(s) \dd s\right)
\end{equation}
can be viewed as a ``vanishingly smooth'' best response to the empirical distribution of play of one's opponent.%
\footnote{It is ``smooth'' because the player is employing a regularized best response map (as opposed to a hard $\argmax$ correspondence), and it is ``vanishingly smooth'' because the factor $t\temp(t)$ hardens to $\infty$ as $t\to\infty$ (for a more detailed discussion, see \citealp{BF13}).}
As such, \eqref{eq:USRL} can be seen as a stochastically perturbed variant of vanishingly smooth fictitious play in continuous time, in which case Corollary \ref{cor:reg-rate} provides the continuous-time, stochastic analogue of Theorem 1.8 of \cite{BF13}.

\end{remark}

\section{Extinction of dominated strategies}
\label{sec:dominated}

A fundamental rationality requirement for any game-theoretic learning process is the elimination of suboptimal, dominated strategies.
Formally, given a finite game $\game\equiv\game(\play,\act,\pay)$, we say that \emph{$p_{k}\in\strat_{k}$ is dominated by $p_{k}'\in\strat_{k}$} (and we write $p_{k} \prec p_{k}'$) if
\begin{equation}
\label{eq:dom}
\smallbraket{\payv_{k}(x)}{p_{k}}
	< \smallbraket{\payv_{k}(x)}{p_{k}'}
	\quad
	\text{for all $x\in\strat$.}
\end{equation}
Strategies that are \emph{iteratively dominated}, \emph{undominated}, or \emph{iteratively undominated} are defined similarly;
also, for \emph{pure} strategies $\alpha,\beta\in\act_{k}$, we obviously have $\alpha\prec\beta$ if and only if
\begin{equation}
\label{eq:dom-pure}
\payv_{k\alpha}(x)
	< \payv_{k\beta}(x)
	\quad
	\text{for all $x\in\strat$.}
\end{equation}
With this in mind, given a trajectory of play $x(t)$, $t\geq0$, we will say that a pure strategy $\alpha\in\act_{k}$ \emph{becomes extinct along $x(t)$} if $x_{k\alpha}(t) \to 0$ as $t\to\infty$.
More generally, following \cite{SZ92}, we will say that the mixed strategy $p_{k}\in\strat_{k}$ becomes extinct along $x(t)$ if $\min\{x_{k\alpha}(t):\alpha\in\supp(p_{k})\}\to0$;
otherwise, $p_{k}$ is said to \emph{survive}.

In the case of perfect payoff observations, \cite{MS16} showed that dominated strategies become extinct under the general reinforcement learning dynamics \eqref{eq:RL}, thus extending earlier results for the replicator dynamics (for an overview, see \citealp{Vio15}).
Under noise and uncertainty however, the situation is a bit more complex:
in the replicator dynamics with aggregate shocks \eqref{eq:ASRD}, \cite{Cab00}, \cite{Imh05} and \cite{HI09} provided a set of sufficient conditions on the intensity of the noise that guarantee the elimination of dominated strategies;
on the other hand, \cite{MM10} showed that the noisy replicator dynamics \eqref{eq:SRD} eliminate all strategies that are not iteratively undominated, with no conditions on the noise level.

As we show below, this unconditional elimination result is a consequence of the core reinforcement principle behind \eqref{eq:RL} and it extends to the entire class of learning dynamics covered by \eqref{eq:SRL}:

\begin{theorem}
\label{thm:dom}
Let $X(t)$ be a solution orbit of \eqref{eq:SRL}.
If $p_{k}\in\strat_{k}$ is dominated \textup(even iteratively\textup),
then it becomes extinct along $X(t)$ almost surely.
\end{theorem}

The  basic idea of our proof is to show that the Itô process $X(t)$ has a dominant drift term that pushes it away from dominated strategies.
However, given the complicated form of the (non-autonomous) dynamics \eqref{eq:evolution}, we do so by using the Fenchel coupling \eqref{eq:Fenchel} to relate the evolution of the generating process $Y_{k}(t)$ to $p_{k}$.

\begin{proof}[Proof of Theorem \ref{thm:dom}]
Suppose $p_{k} \prec p_{k}'$ so that $\smallbraket{\payv_{k}(x)}{p_{k}' - p_{k}} \geq m_{k}$ for some $m_{k}>0$ and for all $x\in\strat$.
Then, if $Y(t)$ is a solution orbit of \eqref{eq:SRL}, we get:
\begin{flalign}
\label{eq:Yest0}
d\braket{Y_{k}}{p_{k}' - p_{k}}
	= \braket{dY_{k}}{p_{k}' - p_{k}}
	&= \braket{\payv_{k}}{p'_{k} - p_{k}} dt
	+ \insum_{\beta}^{k} (p_{k\beta}' - p_{k\beta})\,\sigma_{k\beta} \dd W_{k\beta}
	\notag\\
	&\geq m_{k} \dd t + \insum_{\beta}^{k} (p_{k\beta}' - p_{k\beta})\,\sigma_{k\beta} \dd W_{k\beta},
\end{flalign}
or, equivalently:
\begin{equation}
\label{eq:Yest1}
\braket{Y_{k}(t)}{p_{k}' - p_{k}}
	\geq c_{k} + m_{k} t + \xi_{k}(t),
\end{equation}
where $c_{k} = \braket{Y_{k}(0)}{p_{k} - p_{k}'}$ and
\begin{equation}
\label{eq:diffnoise}
\xi_{k}(t)
	= \insum_{\beta}^{k} (p_{k\beta} - p_{k\beta}') \int_{0}^{t} \sigma_{k\beta}(X(s)) \dd W_{k\beta}(s).
\end{equation}

Consider now the rate-adjusted ``cross-coupling''
\begin{flalign}
\label{eq:cross}
V_{k}(y_{k})
	&= \temp_{k}^{-1} \left[ \fench_{k}(p_{k},\temp_{k} y_{k}) - \fench_{k}(p_{k}',\temp_{k}y_{k}) \right]
	\notag\\
	&= \temp_{k}^{-1} \left[h_{k}(p_{k}) - h_{k}(p_{k}') \right] - \braket{y_{k}}{p_{k} - p_{k}'},
\end{flalign}
with $\fench_{k}$ defined as in \eqref{eq:Fenchel}.
Then, by substituting \eqref{eq:Yest1} in \eqref{eq:cross} and recalling that $\fench_{k}(p_{k}',y_{k}) \geq 0$, we obtain:
\begin{equation}
\label{eq:Yest2}
\fench_{k}(p_{k}, \temp_{k} Y_{k})
	\geq h_{k}(p_{k}) - h_{k}(p_{k}')
	+ \temp_{k}\cdot \left[c_{k} + m_{k} t + \xi_{k}(t) \right].
\end{equation}
To proceed, Lemma \ref{lem:Wbound} shows that $m_{k}t + \xi_{k}(t) \sim m_{k}t$ (a.s.), so the RHS of \eqref{eq:Yest2} tends to infinity as $t\to\infty$ on account of the fact that $t\temp_{k}(t) \to +\infty$ (cf. Assumption \ref{asm:temp}).
In turn, this gives $\fench_{k}(p_{k},\temp_{k}(t)Y_{k}(t))\to+\infty$, so $p_{k}$ becomes extinct along $X(t) = \choice(\temp(t) Y(t))$ by virtue of Proposition \ref{prop:Fenchel}.
Finally, our claim for iteratively dominated strategies follows by induction on the rounds of elimination of dominated strategies \textendash\ cf. \citet[Proposition 1A]{Cab00}.
\end{proof}

Theorem \ref{thm:dom} shows that dominated strategies become extinct under \eqref{eq:SRL} but it does not give any information on the rate of extinction \textendash\ or how probable it is to observe a dominated strategy above a given level at some $t\geq0$.
To address this issue, we provide two results below:
Proposition~\ref{prop:dom-asymptotics} describes the long-term decay rate of dominated strategies and provides a ``large deviations'' bound for the probability of observing a dominated strategy above a given level at time $t\geq0$.
Subsequently, in Proposition~\ref{prop:dom-mean} we estimate the expected time it takes for a dominated strategy to drop below a given level.
For simplicity, we present our results in the case where the players' choice maps are derived from separable penalty functions as in \eqref{eq:decomposable};
again, proofs are relegated to Appendix \ref{app:main}:

\begin{proposition}
\label{prop:dom-asymptotics}
Let $\alpha\in\act_{k}$ be dominated by $\beta\in\act_{k}$ and assume that the choice map of player $k$ is generated by a separable penalty function of the form \eqref{eq:decomposable} with $\lim_{x\to0^{+}} \theta_{k}'(x) = -\infty$.
Then, for all $\delta,\eps>0$ and for all large enough $t$, we have:
\begin{flalign}
\label{eq:dom-rate}
&X_{k\alpha}(t)
	\leq \phi_{k}\left[ C_{k} - \temp_{k}(t) \left( m_{k}t  - 2(1+\eps) \sigma_{\alpha\beta} \sqrt{t \log \log t} \right) \right]
	\quad
	\text{\textup(a.s.\textup)}
\intertext{and}
\label{eq:dom-bound}
&\prob\left(X_{k\alpha}(t) > \delta\right)
	\leq \frac{1}{2}
	\erfc\left[
	\frac{1}{2\sigma_{\alpha\beta}}
	\left(
	m_{k}\sqrt{t} - \frac{C_{k} - \theta_{k}'(\delta)}{\temp_{k}(t) \sqrt{t}}
	\right)
	\right]
\end{flalign}
where:
\begin{enumerate}
[\indent\itshape a\upshape)]
\setlength{\itemsep}{0pt}
\item
$\erfc(z) = \frac{2}{\sqrt{\pi}} \int_{z}^{\infty} e^{-t^{2}} dt$ is the complementary error function.
\item
$\phi_{k} = (\theta_{k}')^{-1}$ \textup(note that $\lim_{z\to-\infty} \phi_{k}(z) = 0$ by assumption\textup).
\item
$m_{k} = \min_{x\in\strat}\{\payv_{k\beta}(x) - \payv_{k\alpha}(x)\} >0$ is the minimum payoff difference between $\alpha$ and $\beta$.
\item
$\sigma_{\alpha\beta}^{2} = \frac{1}{2}\max_{x\in\strat}\big\{\sigma_{k\alpha}^{2}(x) + \sigma_{k\beta}^{2}(x)\big\} > 0$.
\item
$C_{k}$ is a constant that depends only on the initial conditions of \eqref{eq:SRL}.
\end{enumerate}
\end{proposition}

\begin{proposition}
\label{prop:dom-mean}
With notation as in Proposition \ref{prop:dom-asymptotics}, assume that \eqref{eq:SRL} is run with constant learning rates $\temp_{k}$ and noisy observations with constant variance.
If $\tau_{\delta} = \inf\{t>0: X_{k\alpha}(t) \leq \delta\}$, then:
\begin{equation}
\label{eq:dom-mean}
\ex[\tau_{\delta}]
	\leq \frac{\big[C_{k} - \theta_{k}'(\delta)\big]_{+}}{\temp_{k} m_{k}}.
\end{equation}
\end{proposition}

Propositions \ref{prop:dom-asymptotics} and \ref{prop:dom-mean} are our main results concerning the rate of elimination of dominated strategies under uncertainty, so a few remarks are in order:

\begin{remark}[\emph{Asymptotics versus mean behavior}]
Propositions \ref{prop:dom-asymptotics} and \ref{prop:dom-mean} capture complementary aspects of the statistics of the extinction rate of dominated strategies under \eqref{eq:SRL}.
For instance, \eqref{eq:dom-rate} describes the asymptotic rate of elimination of dominated strategies but it does not provide an estimate of how much time must pass until this rate becomes relevant.
On the other hand, the bound \eqref{eq:dom-mean} estimates the mean time it takes for a dominated strategy to fall below a given level;
however, in contrast to \eqref{eq:dom-bound}, it does not describe how probable it is to observe deviations from this mean.
\end{remark}

\begin{remark}[\emph{Comparison with the noiseless case}]
When $\sigma=0$, the decay estimate \eqref{eq:dom-rate} boils down to $\phi_{k}[C_{k} - m_{k}\temp_{k}(t) t]$, a bound which recovers the results of \cite{MS16} for the deterministic dynamics \eqref{eq:RL}.
Thus, even though the noise may initially mask the fact that a strategy is dominated, Proposition \ref{prop:dom-asymptotics} shows that the long-run behavior of \eqref{eq:SRL} and \eqref{eq:RL} is the same as far as dominated strategies are concerned.
This is in contrast with the aggregate-shocks dynamics \eqref{eq:ASRD} where, even if a dominated strategy becomes extinct, its rate of elimination is different to leading order than in the noiseless case \citep[Theorem 3.1]{Imh05}.
\end{remark}

\begin{remark}[\emph{The role of $\temp$}]
The estimate \eqref{eq:dom-rate} shows that running \eqref{eq:SRL} with large, constant $\temp$ leads to a much faster rate of extinction of dominated strategies.
In particular, recalling that $\erfc(t) \sim \pi^{-1/2} t^{-1} e^{-t^{2}}$ as $t\to\infty$, the bound \eqref{eq:dom-bound} becomes:
\begin{equation}
\label{eq:dom-decay}
\prob(X_{k\alpha}(t) > \delta)
	= \bigoh\left( t^{-1/2} \exp\left(-\frac{m_{k}^{2} t}{2\sigma_{\alpha\beta}^{2}}\right) \right),
\end{equation}
up to a subleading term of the order of $\bigoh(1/\temp_{k}^{2}(t))$ in the exponent.
Thus, even though the leading behavior of \eqref{eq:dom-bound} is not affected by the player's choice of learning parameter, the subleading term is minimized for large, constant $\temp$ so the asymptotic bound \eqref{eq:dom-decay} becomes tighter in that case.

We thus observe an important contrast between regret minimization and the elimination of dominated strategies:
whereas the optimal regret guarantee of Theorem \ref{thm:consistent} is achieved for $\temp(t) \propto 1/\sqrt{t}$, the asymptotic extinction rate of dominated strategies is much faster for constant $\temp$.
The reason for this disparity is that higher values of $\temp$ reinforce consistent payoff differences and therefore eliminate dominated strategies faster (independently of the noise level).
On the other hand, to attain a no-regret state, players must be careful not to make too many mistakes in the presence of noise, so a more conservative choice of $\temp$ is warranted.
\end{remark}


\section{Long-term stability and convergence analysis}
\label{sec:folk}

We now turn to the long-term stability and convergence properties of the dynamics \eqref{eq:SRL} with respect to equilibrium play.
To that end, recall first that $\eq\in\strat$ is a \emph{Nash equilibrium} of $\game\equiv\game(\play,\act,\pay)$ if it is unilaterally stable for all players, i.e.
\begin{equation}
\label{eq:Nash}
\pay_{k}(\eq)
	\geq \pay_{k}(x_{k},\eq_{-k})
	\quad
	\text{for all $x_{k}\in\strat_{k}$, $k\in\play$,}
\end{equation}
or, equivalently:
\begin{equation}
\label{eq:Nash-coords}
\payv_{k\alpha}(\eq) \geq \payv_{k\beta}(\eq)
	\quad
	\text{for all $k\in\play$ and for all $\alpha\in\supp(\eq_{k})$, $\beta\in\act_{k}$.}
\end{equation}
\emph{Strict} equilibria are defined by requiring that \eqref{eq:Nash} hold as a strict inequality for all $x_{k}\neq \eq_{k}$;
obviously, such equilibria are also \emph{pure} in the sense that they correspond to pure strategy profiles in $\act=\prod_{k}\act_{k}$ (i.e. vertices of $\strat$).

In the noiseless case ($\sigma=0$) with constant learning rates ($\dot\temp =0$), \cite{MS16} recently showed that the deterministic dynamics \eqref{eq:RL} exhibit the following properties with respect to Nash equilibria of $\game$:
\begin{enumerate}
[\indent\itshape a\upshape)]
\setlength{\itemsep}{0pt}

\item
If a solution orbit of \eqref{eq:RL} converges to $\eq$, then $\eq$ is a Nash equilibrium.

\item
If $\eq\in\strat$ is (Lyapunov) stable, then it is also a Nash equilibrium.

\item
Strict Nash equilibria are asymptotically stable in \eqref{eq:RL}.
\end{enumerate}

In turn, these properties are generalizations of the long-term stability and convergence properties of the (multi-population) replicator dynamics \textendash\ sometimes referred to as the ``folk theorem'' of evolutionary game theory \citep{HS98,HS03}.
That being said, the situation is quite different in the presence of noise:
for instance, interior Nash equilibria are not even traps (almost sure rest points) of the stochastic reinforcement learning dynamics \eqref{eq:SRL}, so the ordinary (deterministic) definitions of stability and convergence no longer apply.
Instead, in the context of stochastic differential equations, Lyapunov and asymptotic stability are defined as follows \citep{Kha12}:

\begin{definition}
\label{def:stability}
Let $\eq\in\strat$.
We will say that:
\begin{enumerate}
\item
$\eq$ is \emph{stochastically \textup(Lyapunov\textup) stable} under \eqref{eq:SRL} if, for every $\eps>0$ and for every neighborhood $U_{0}$ of $\eq$ in $\strat$, there exists a neighborhood $U\subseteq U_{0}$ of $\eq$ such that
\begin{equation}
\label{eq:stable-Lyap}
\prob(\text{$X(t)\in U_{0}$ for all $t\geq0$})
	\geq 1-\eps,
\end{equation}
whenever $X(0) \in U$.


\item
$\eq$  is \emph{stochastically asymptotically stable} under \eqref{eq:SRL} if it is stochastically stable and attracting:
for every $\eps>0$ and for every neighborhood $U_{0}$ of $\eq$ in $\strat$, there exists a neighborhood $U\subseteq U_{0}$ of $\eq$ such that
\begin{equation}
\label{eq:stable-asym}
\txs
\prob\left(
	\text{$X(t)\in U_{0}$ for all $t\geq0$ and $\lim_{t\to\infty} X(t) = \eq$}
	\right)
	\geq 1-\eps,
\end{equation}
whenever $X(0) \in U$.
\end{enumerate}
\end{definition}

In the evolutionary setting of the stochastic replicator dynamics with aggregate shocks, \cite{Imh05} and \cite{HI09} showed that strict Nash equilibria are stochastically asymptotically stable under \eqref{eq:ASRD} provided that the variability of the shocks across different strategies is small enough.
More recently, in a learning context, \cite{MM10} showed that the same holds for the stochastic replicator dynamics \eqref{eq:SRD} of exponential learning (with constant $\temp$), irrespective of the variance of the observation noise.
However, this last result relies heavily on the specific properties of the logit map \eqref{eq:logit} and the infinitesimal generator of \eqref{eq:SRD}.

In our case, the convoluted (and non-autonomous) form of the stochastic dynamics \eqref{eq:evolution} complicates things considerably, so such an approach is not possible \textendash\ especially with regards to finding a local stochastic Lyapunov function for the dynamical system \eqref{eq:evolution} that governs the evolution of $X(t)$.
Nonetheless, by working \emph{directly} with \eqref{eq:SRL}, we obtain the following general result:

\begin{theorem}
\label{thm:folk}
Let $X(t)$ be a solution orbit of \eqref{eq:SRL} and let $\eq\in\strat$.
Then:
\begin{enumerate}
[\indent\upshape(1)]
\setlength{\itemsep}{0pt}
\item
If\;$\prob\left(\lim_{t\to\infty} X(t) = \eq\right) > 0$, $\eq$ is a Nash equilibrium of $\game$.
\item
If $\eq$ is stochastically \textup(Lyapunov\textup) stable, it is also Nash.
\item
If $\eq$ is a strict Nash equilibrium of $\game$, it is stochastically asymptotically stable under \eqref{eq:SRL}.
\end{enumerate}
\end{theorem}


The fist ingredient of our proof (presented in detail in Appendix \ref{app:main}) is to show that if the process $X(t)$ remains in the vicinity of $\eq\in\strat$ for all $t\geq0$ with positive probability, then $\eq$ must be a Nash equilibrium.
This is formalized in the following proposition (which is of independent interest):

\begin{proposition}
\label{prop:stability}
With notation as in Theorem \ref{thm:folk}, assume that every neighborhood $U$ of $\eq$ in $\strat$ admits with positive probability a solution orbit $X(t)$ of \eqref{eq:SRL} such that $X(t) \in U$ for all $t\geq0$.
Then, $\eq$ is a Nash equilibrium.
\end{proposition}

\begin{proof}
If $\eq$ is not Nash, we must have $\payv_{k\alpha}(\eq) < \payv_{k\beta}(\eq)$ for some player $k\in\play$ and for some $\alpha\in\supp(\eq_{k})$, $\beta\in\act_{k}$.
On that account, let $U$ be a sufficiently small neighborhood of $\eq$ in $\strat$ such that $\payv_{k\beta}(x) - \payv_{k\alpha}(x) \geq m_{k}$ for some $m_{k} > 0$ and for all $x\in U$.
Then, conditioning on the positive probability event that there exists an orbit $X(t) = \choice(\temp(t) Y(t))$ of \eqref{eq:SRL} that is contained in $U$ for all $t\geq0$, we have:
\begin{flalign}
\label{eq:Ydiff1}
dY_{k\alpha} - dY_{k\beta}
	&= \left(\payv_{k\alpha}(X) - \payv_{k\beta}(X)\right) \dd t
	+ \sigma_{k\alpha} \dd W_{k\alpha} - \sigma_{k\beta} \dd W_{k\beta}
	\notag\\
	&\leq -m_{k} \dd t - d\xi_{k},
\end{flalign}
where
$\xi_{k}(t) = \int_{0}^{t} \sigma_{k\beta}(X(s)) \dd W_{k\beta}(s) - \int_{0}^{t} \sigma_{k\alpha}(X(s)) \dd W_{k\alpha}(s)$.
Since the diffusion coefficients $\sigma$ are bounded, the process $\xi_{k}(t)$ will grow at most as $\bigoh(\sqrt{t\log\log t})$, so we get $\xi_{k}(t) + m_{k}t \sim m_{k}t$ (for a precise statement, see Lemma \ref{lem:Wbound}).
Thus, on account of Assumption \ref{asm:temp} for $\temp_{k}(t)$, we obtain:
\begin{equation}
\temp_{k}(t) \cdot \big[ Y_{k\alpha}(t) - Y_{k\beta}(t) \big]
	\leq -\frac{m_{k} t + \xi_{k}(t)}{t} \cdot \temp_{k}(t) t
	\to -\infty,
\end{equation}
The above shows that the score difference between $\alpha$ and $\beta$ grows to $-\infty$, so $X_{k\alpha}(t)$ tends to zero as $t\to\infty$ (Proposition \ref{prop:Fenchel}).
This contradicts the assumption that $X(t) \in U$ for all $t$ (recall that $\alpha\in\supp(\eq_{k})$), so $\eq$ must be a Nash equilibrium of $\game$.
\end{proof}

With Proposition \ref{prop:stability} at hand, we are left to show that strict equilibria are stochastically asymptotically stable under \eqref{eq:SRL}.
This is carried out in Appendix \ref{app:main} by showing that
\begin{inparaenum}
[\itshape a\upshape)]
\item
trajectories that remain close to a strict equilibrium are eventually attracted to it (because it is locally dominant);
and
\item
that this occurs with controllably high probability (a result which is harder to establish and which relies on an application of Girsanov's theorem to estimate the probability that a Brownian motion with positive drift attains a given negative level in finite time).
\end{inparaenum}

We close this section with a few remarks on Theorem \ref{thm:folk}:


\begin{remark}[\emph{Convergence rates and the role of $\temp$}]
By reasoning as in Proposition \ref{prop:dom-asymptotics}, it is also possible to show that the rate of convergence of \eqref{eq:SRL} to strict equilibria increases with the players' learning parameter $\temp$.
As in the case of dominated strategies, the reason for this is that players experience a consistent payoff difference in favor of their equilibrium strategy when close to a strict equilibrium, and this difference (which is reflected in the score process $Y(t)$) is heightened by $\temp$.
In this regard, convergence to a strict equilibrium seems to be somewhat antagonistic to attaining a no-regret state in the presence of noise:
the regret guarantee \eqref{eq:reg-bound} is optimized for $\temp(t) \propto t^{-1/2}$ while the dynamics' rate of convergence to strict equilibria is faster for large, constant $\temp$.
\end{remark}

\begin{remark}
\label{rem:modified}
By considering the modified game $\tilde\game$ with noise-adjusted payoff functions $\tilde\payv_{k\alpha}(x) = \payv_{k\alpha}(x) - \frac{1}{2} \sigma_{k\alpha}^{2}$, the reasoning in the proof of Theorem~\ref{thm:folk} yields an alternative proof of the stability and convergence results of \cite{HI09} for the replicator dynamics with aggregate shocks \eqref{eq:ASRD}.
\end{remark}


\section{Time averages in $2$-player games}
\label{sec:averages}

In this section, we analyze the asymptotic behavior of the time averages
\begin{equation}
\label{eq:average}
\bar X(t)
	= \frac{1}{t} \int_{0}^{t} X(s) \dd s,
\end{equation}
and we establish an averaging principle for \eqref{eq:SRL} in $2$-player games.
Our analysis is motivated by the original deterministic results of \citet{HS98} and \cite{HSV09} who showed that $\bar X(t)$ converges to Nash equilibrium under the replicator dynamics
whenever $\liminf_{t\to\infty} X_{k\alpha}(t) > 0$ for all $\alpha\in\act_{k}$, $k=1,2$.
More recently, \cite{MS16} proved a version of this result for arbitrary regularized best response maps (always in a deterministic setting), while \cite{HI09} showed that time-averages of the aggregate-shocks replicator dynamics \eqref{eq:ASRD} converge to the Nash set of a modified game.

Our main result here is that the averaging principle of \cite{HS98} extends to the perturbed reinforcement learning scheme \eqref{eq:SRL}, even in the presence of arbitrarily large payoff observation errors:

\begin{theorem}
\label{thm:averages}
Let $\game$ be a $2$-player game and let $X(t)$ be a solution orbit of the stochastic dynamics \eqref{eq:SRL}.
If the players' score differences $Y_{k\alpha}(t) - Y_{k\beta}(t)$ grow sublinearly with $t$ for all $\alpha,\beta\in\act_{k}$, $k=1,2$, then the time-averaged process $\bar X(t)$ converges almost surely to the set of Nash equilibria of $\game$.
\end{theorem}

\begin{proof}[Proof of Theorem \ref{thm:averages}]
Pick $\alpha,\beta\in\act_{k}$.
Then, by the definition of \eqref{eq:SRL}, we get:
\begin{flalign}
\label{eq:Yavg1}
Y_{k\alpha}(t) - Y_{k\beta}(t)
	&= Y_{k\alpha}(0) - Y_{k\beta}(0)
	+ \int_{0}^{t} \left[ \payv_{k\alpha}(X(s)) - \payv_{k\beta}(X(s)) \right] \dd s
	\notag\\
	&+ \int_{0}^{t} \sigma_{k\alpha}(X(s)) \dd W_{k\alpha}(s)
	- \int_{0}^{t} \sigma_{k\beta}(X(s)) \dd W_{k\beta}(s)
	\notag\\
	&= c_{k} + \xi_{k}(t)
	+ t \left[ \payv_{k\alpha}(\bar X(t)) - \payv_{k\beta}(\bar X(t)) \right],
\end{flalign}
where $c_{k} = Y_{k\alpha}(0) - Y_{k\beta}(0)$, $\xi_{k}(t)$ denotes the martingale part of \eqref{eq:Yavg1} and we have used the fact that $\payv_{k}(x)$ is linear in $x$ (and not only multilinear).
Since $\xi_{k}(t)$ is subleading in \eqref{eq:Yavg1} with respect to $t$ (cf. Lemma \ref{lem:Wbound} for a precise statement), dividing by $t$ yields:
\begin{equation}
\label{eq:Yavg2}
\lim_{t\to\infty} \left[ \payv_{k\alpha}(\bar X(t)) - \payv_{k\beta}(\bar X(t)) \right]
	= 0
	\quad
	\text{(a.s.),}
\end{equation}
so $\payv_{k\alpha}(\eq) = \payv_{k\beta}(\eq)$ whenever $\eq$ is an $\omega$-limit of $\bar X(t)$.
This shows that any $\omega$-limit of $\bar X(t)$ is a Nash equilibrium, and since the $\omega$-set of $\bar X(t)$ is nonempty (by compactness of $\strat$), our claim follows. 
\end{proof}

\begin{figure}[t]
\subfigure[A sample trajectory and its time average.]{
\label{fig:averages-sample}
\includegraphics[width=.485\textwidth]{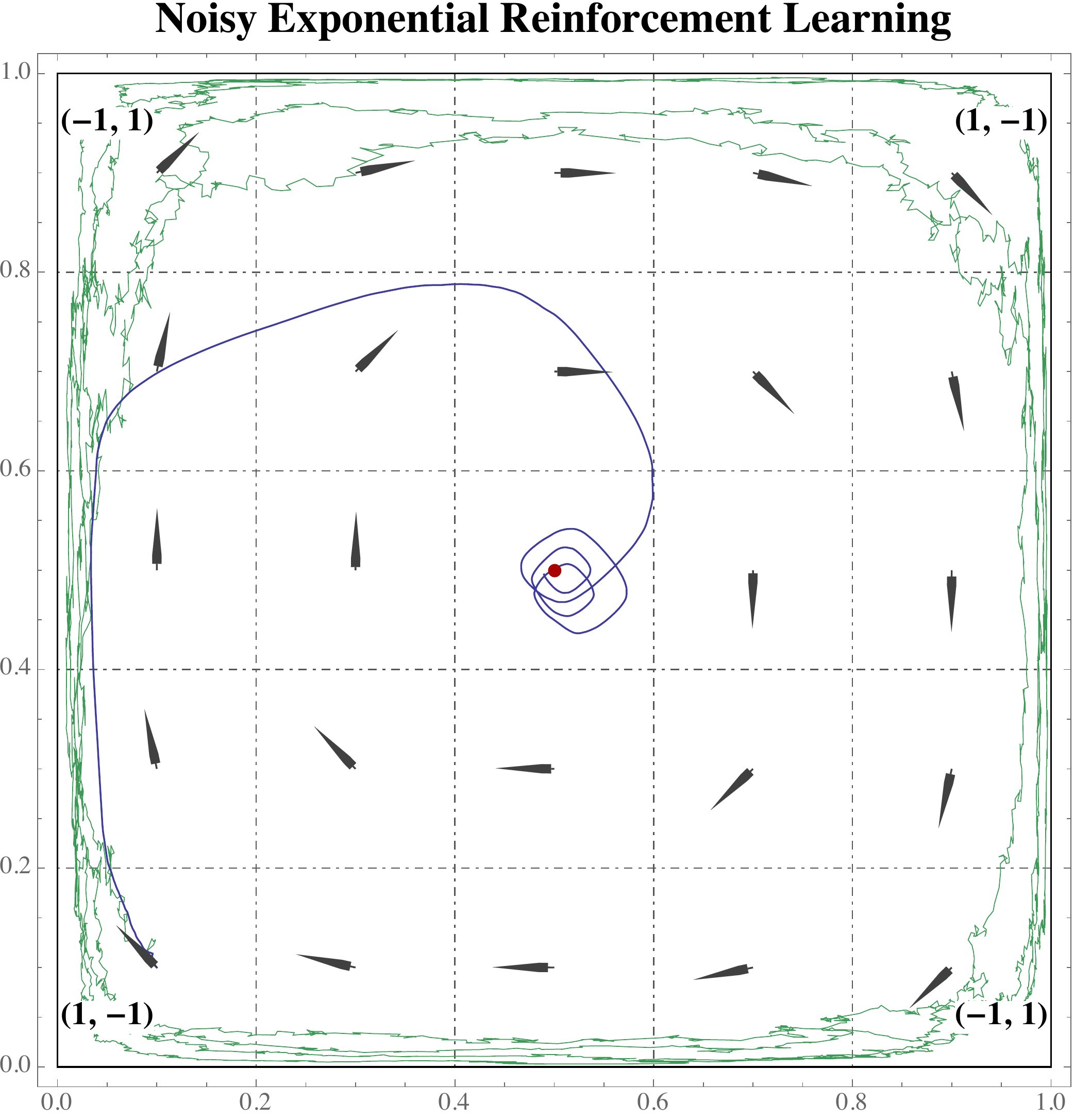}}
\hfill
\subfigure[Distribution of time averages at time $T$.]{
\label{fig:averages-density}
\shortstack{
\includegraphics[width=.23\textwidth]{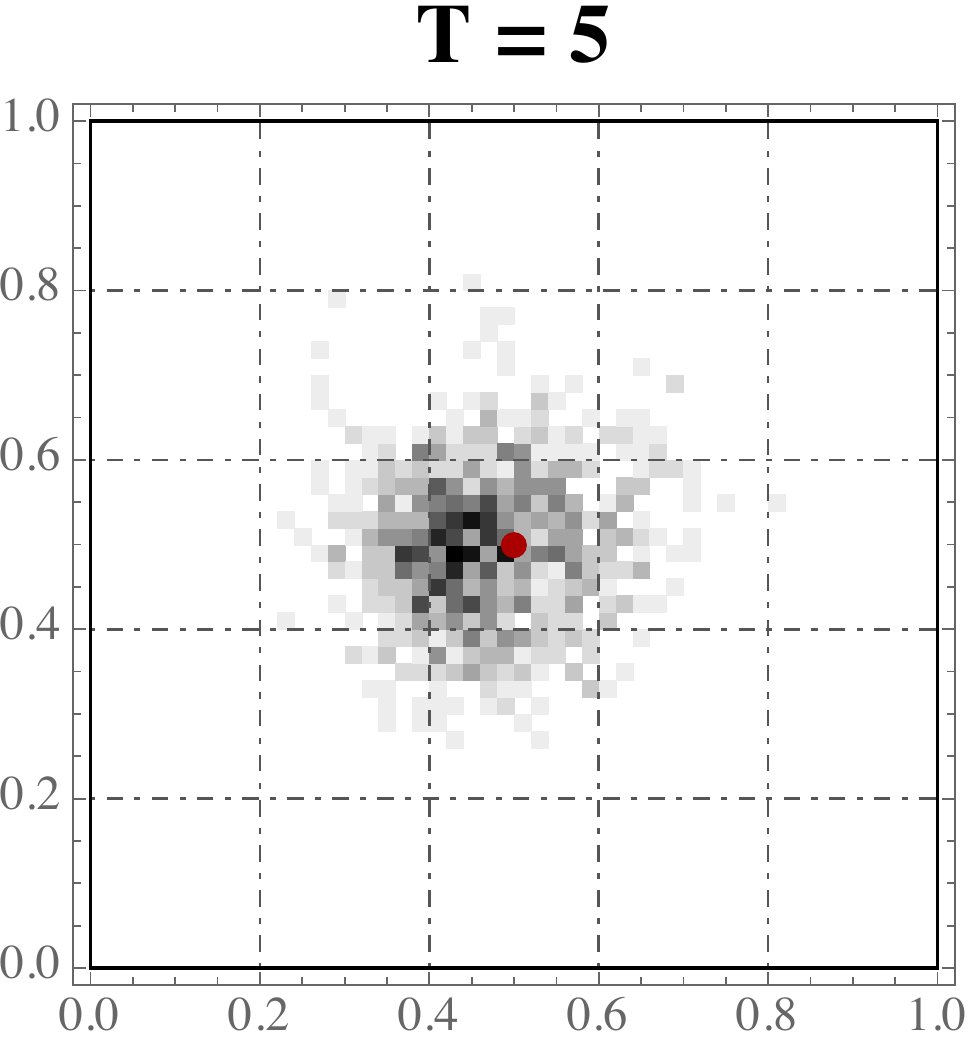}
\includegraphics[width=.23\textwidth]{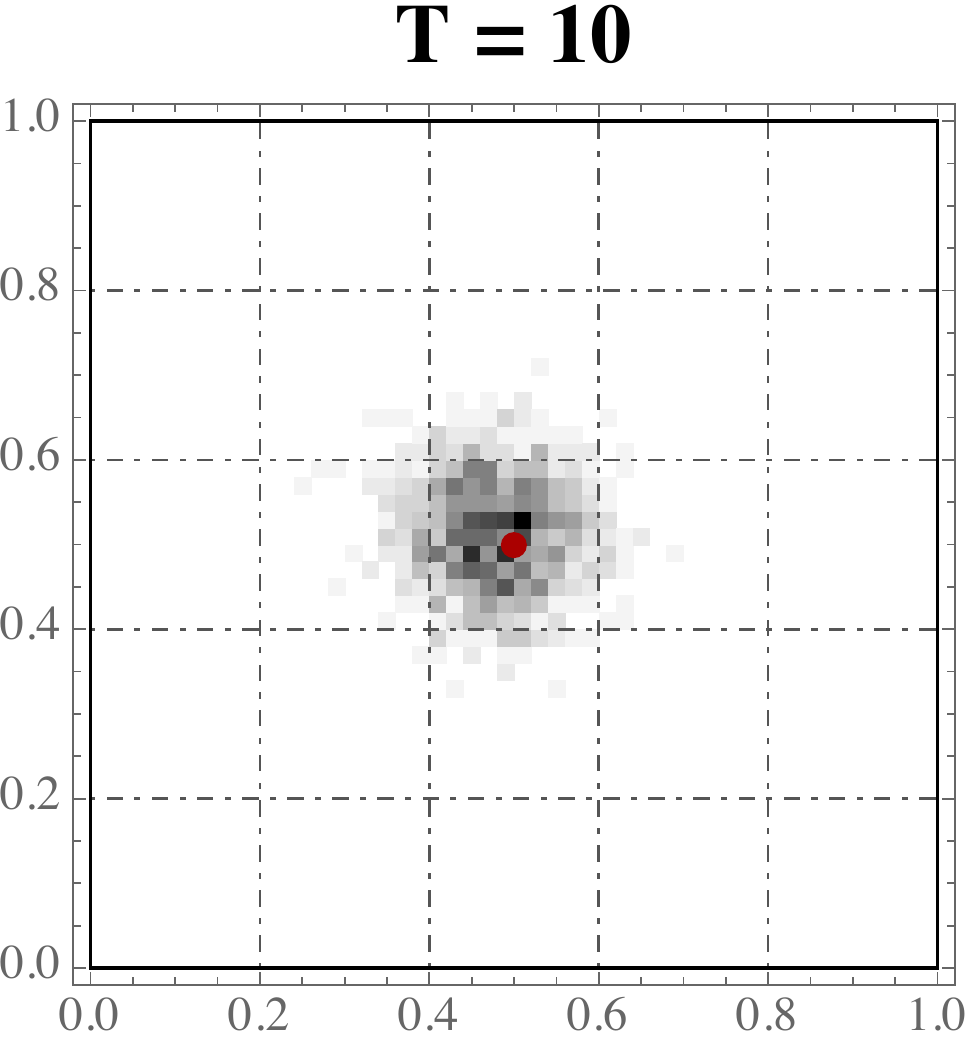}
\\
\includegraphics[width=.23\textwidth]{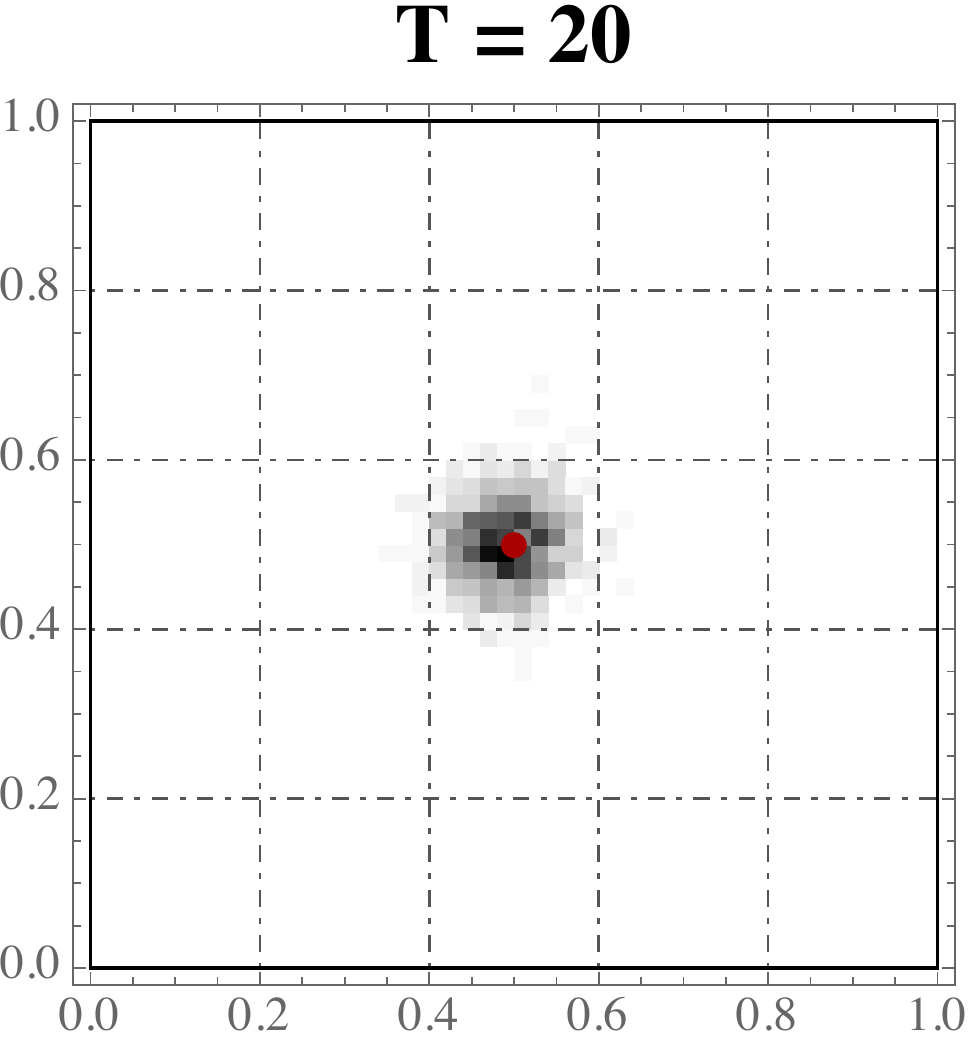}
\includegraphics[width=.23\textwidth]{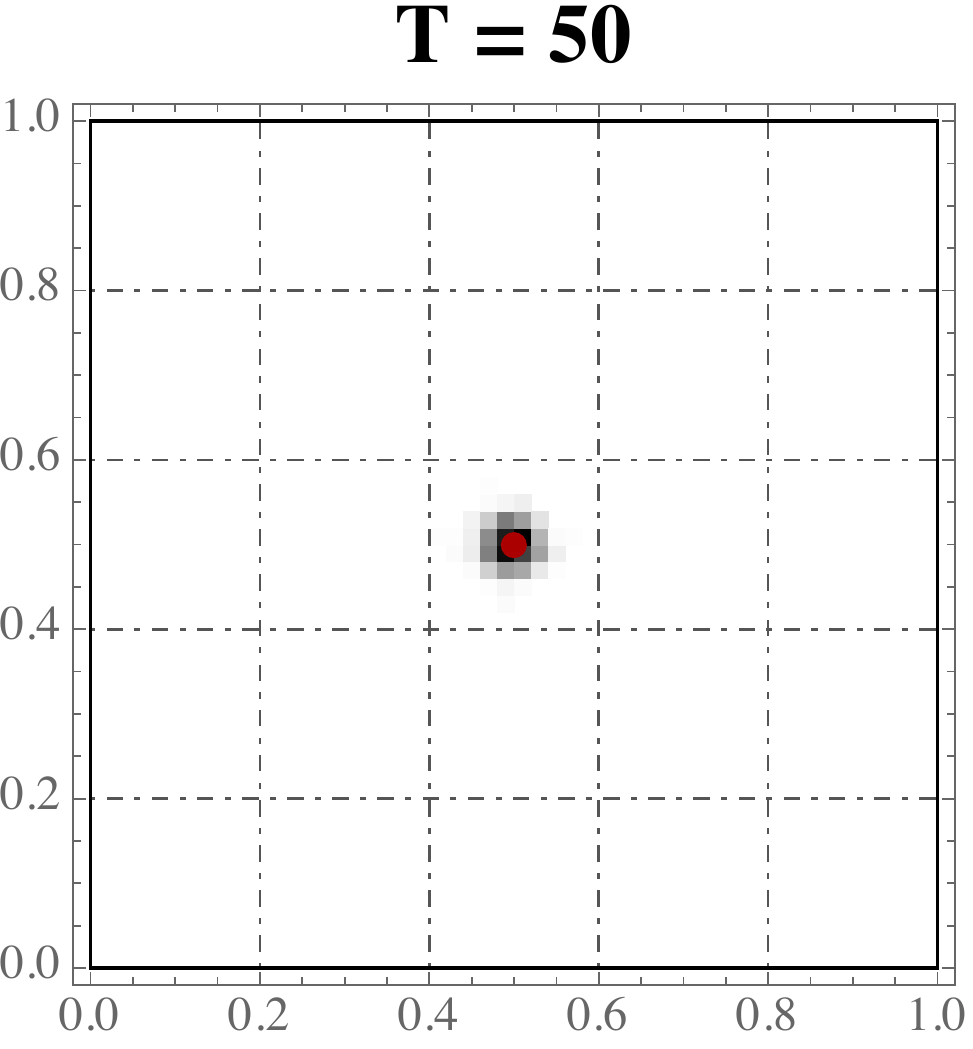}
} 
} 

\caption{\footnotesize
Time averages of \eqref{eq:SRL} with logit best responses in a game of Matching Pennies \textup(as in Fig.~\ref{fig:portraits}, Nash equilibria are depicted in red and the game's payoff's are displayed inline; for benchmarking purposes, we also took $\sigma_{k\alpha}=1$ for all $\alpha\in\act_{k}$, $k=1,2$\textup).
Fig.~\ref{fig:averages-sample} shows the evolution of a sample trajectory and its time average;
in Fig.~\ref{fig:averages-density}, we show a density plot of the distribution of $10^{4}$ time-averaged trajectories for different values of the integration horizon $T$.
In tune with Proposition \ref{prop:zero-sum}, we see that time averages converge to the game's Nash equilibrium.%
}
\label{fig:averages}
\end{figure}

In the case of \eqref{eq:EW}/\eqref{eq:RD}, the sublinear growth requirement for $Y_{k\alpha} - Y_{k\beta}$ boils down to the permanency condition $\liminf_{t\to\infty} X_{k\alpha}(t) > 0$, so we recover the original result of \cite{HS98}.
Other than that however, the applicability of Theorem \ref{thm:averages} is limited by the growth requirement for $Y_{k\alpha}(t) - Y_{k\beta}(t)$.
The following proposition (whose proof can be found in Appendix~\ref{app:main}) shows that this condition always holds in $2$-player zero-sum games under the additional assumption that $\lim_{t\to\infty} \temp_{k}(t) = 0$:

\begin{proposition}
\label{prop:zero-sum}
Let $\game$ be a $2$-player zero-sum game with an interior Nash equilibrium, and assume that \eqref{eq:SRL} is run with $\temp_{k}(t)$ satisfying $\lim_{t\to\infty} \temp_{k}(t) = 0$ \textup(in addition to Assumption \ref{asm:temp}\textup).
Then, the time average of \eqref{eq:SRL} converges to the set of Nash equilibria of $\game$ \textup(a.s.\textup).
\end{proposition}


Proposition \ref{prop:zero-sum} is reminiscent of Theorem \ref{thm:consistent} on the no-regret properties of \eqref{eq:SRL} in that both results require a vanishing learning parameter.
In the noiseless regime however ($\sigma=0$), this is not the case:
in the replicator dynamics (which correspond to \eqref{eq:EW} with $\temp=1$), the conclusions of Proposition \ref{prop:zero-sum} still hold essentially verbatim \citep{HS98}.
Intuitively, the reason for this is that time-averaged replicator orbits exhibit the same long-term behavior as the (deterministic) best response dynamics of \cite{GM91}:
\begin{equation}
\label{eq:BRD}
\tag{BRD}
\dot x_{k}
	\in \brep_{k}(x) - x_{k},
\end{equation}
where $\brep_{k}(x) \equiv \argmax_{x_{k}'\in\strat_{k}} \smallbraket{\payv_{k}(x)}{x_{k}'}$ denotes the standard (non-regularized) best response correspondence of player $k$.
Thus, given that the dynamics \eqref{eq:BRD} converge to equilibrium in zero-sum games, the same must also hold for the time averages of \eqref{eq:RD}.

More precisely, \cite{HSV09} showed that the $\omega$-limit set $\Omega$ of time-averaged solutions of \eqref{eq:RD} is \emph{\acl{ICT}} (\acs{ICT}) under \eqref{eq:BRD},
\acused{ICT}
i.e. any two points in $\Omega$ may be joined by a piecewise continuous ``chain'' of arbitrarily long orbit segments of \eqref{eq:BRD} in $\Omega$ broken by arbitrarily small jump discontinuities.%
\footnote{In particular, \ac{ICT} sets are invariant, connected and have no proper attractors;
for a full development, see \citealp{BHS05}).}
\cite{MS16} subsequently established a similar averaging principle linking the deterministic reinforcement learning scheme \eqref{eq:RL} to \eqref{eq:BRD}.
Importantly, as we show below, the \emph{stochastic} dynamics \eqref{eq:SRL} share the same connection to the \emph{deterministic} best response dynamics \eqref{eq:BRD}, despite all the noise:

\begin{theorem}
\label{thm:BRD}
Let $X(t)$ be a solution orbit of \eqref{eq:SRL} for a $2$-player game $\game$.
Then, the $\omega$-limit set of $\bar X(t)$ is \acl{ICT} under the deterministic best response dynamics \eqref{eq:BRD}.
\end{theorem}
\begin{proof}
  Our proof follows the approach of \cite{HSV09} and is presented in detail in Appendix~\ref{app:main}.
\end{proof}

Thanks to Theorem \ref{thm:BRD}, several conclusions of \cite{HSV09} for $2$-player games can be readily generalized to the full \emph{stochastic} setting of \eqref{eq:SRL} simply by exploiting the properties of the \emph{deterministic} dynamics \eqref{eq:BRD}:
\begin{enumerate}
\setlength{\itemsep}{0pt}
\item
If $\bar X(t)$ converges under \eqref{eq:SRL}, its limit is a Nash equilibrium.
\item
Any global attractor of \eqref{eq:BRD} also attracts time averages of \eqref{eq:SRL}, independently of the noise level.
In particular, since the set of Nash equilibria is globally attracting under \eqref{eq:BRD} in zero-sum games, this observation extends Proposition \ref{prop:zero-sum} to the constant $\temp$ case.
\item
The only \ac{ICT} sets of \eqref{eq:BRD} in potential games consist of (isolated) components of Nash equilibria, so $\bar X(t)$ converges to one such component.
\end{enumerate}

\section{Discussion}
\label{sec:discussion}

\subsection{Correlated observation noise}

An important assumption underlying the noisy learning model \eqref{eq:SRL} is that the perturbations to the players' payoff observations are stochastically independent across players and strategies.
However, in many cases of practical interest, such observations are actually \emph{correlated}:
in congestion games for example, a choice of action corresponds to a choice of route in a road network, and two routes that overlap would have the same uncertainty characteristics over common edges/road segments (and similarly for correlations across different players).

To account for correlations of this sort, we can consider the more general learning model 
\begin{equation}
\label{eq:SRL-corr}
\tag{SRL$_{\textup{c}}$}
\begin{aligned}
dY_{k\alpha}
	&= \payv_{k\alpha}(X) \dd t + \dd Z_{k\alpha},
	\\
X_{k}
	&= \choice_{k}(\temp_{k} Y_{k}),
\end{aligned}
\end{equation}
where $Z_{k\alpha}$ is a family of Itô martingale processes with quadratic variation
\begin{equation}
\label{eq:corr}
d[Z_{k\alpha}, Z_{j\beta}]
	= dZ_{k\alpha} \cdot dZ_{k\beta}
	= \Sigma_{\alpha\beta}^{kj} \dd t
	\qquad
	k,j\in\play,
	\alpha\in\act_{k},
	\beta\in\act_{j},
\end{equation}
and such that $\sup_{t} \norm{\Sigma(t)} < 0$.
Obviously, if $\Sigma$ is block-diagonal (i.e. $\Sigma_{\alpha\beta}^{kj} \propto \delta_{\alpha\beta} \delta_{kj}$), we recover the original learning model (SRL) with uncorrelated uncertainty across strategies and payoffs;
otherwise, in the non-diagonal case, the covariance matrix $\Sigma$ simply captures the possible correlations between the players' observations.

A first example of this type is provided by \cite{Cab00} where $Z_{k\alpha}$ is assumed to be of the form $dZ_{k\alpha} = \sum_{r=1}^{d} \sigma_{k\alpha}^{r} \dd W_{r}$ with $W(t)$ denoting a standard Wiener process in $\R^{d}$.
Such correlations also occur naturally in the study of congestion games:
indeed, let $\graph = (\vertices,\edges)$ be a (non-oriented) graph with vertex set $\vertices$ and edge set $\edges$, and let $\play\subseteq\vertices\times\vertices$ denote a set of origin-destination pairs in the network.
In (finite, nonsplittable) congestion games, each player $k$ is typically identified with an origin-destination pair $(v,w)\in\play\subseteq\vertices\times\vertices$, and his actions are drawn from a set $\act_{k}$ of (acyclic) paths in $\graph$ that join $v$ to $w$.
The cost associated to a choice of path $\alpha\in\act_{k}$ is then given by the total time taken to traverse path $\alpha$, i.e. $\loss_{k\alpha} = \sum_{r\in\alpha} \loss_{r} = \sum_{r\in\edges} P_{r\alpha}\loss_{r}$, where $\loss_{r}$ denotes the corresponding delay function of edge $r\in\edges$ (a function of the number of users who use $r\in\edges$) and $P_{r\alpha} = \one\{r\in\alpha\}$ denotes the edge-path incidence matrix of the network.
As a result, if the delay across each edge is only observable up to a random error, the players' score updates would take the form:
\begin{equation}
dY_{k\alpha}
	= - \sum_{r\in\alpha} (\loss_{r}(X) \dd t + \sigma_{r} \dd W_{r})
	= - \loss_{k\alpha}(X) \dd t + \dd Z_{k\alpha},
\end{equation}
with $dZ_{k\alpha} = -\sum_{r\in\edges} P_{r\alpha} \sigma_{r} \dd W_{r}$.
In turn, this leads to the covariance matrix
\begin{flalign}
d[Z_{k\alpha},Z_{j\beta}]
	&= \sum_{r,r'\in\edges} P_{r\alpha} P_{r'\beta} \sigma_{r} \sigma_{r'} \dd W_{r}\cdot dW_{r'}
	\notag\\
	&= \sum_{r\in\edges} P_{r\alpha} P_{r\beta} \sigma_{r}^{2} \dd t
	= \sum_{r\in\alpha\cap\beta} \sigma_{r}^{2} \dd t,
\end{flalign}
i.e. stochastic fluctuations across different paths $\alpha,\beta$ are correlated precisely over the paths' common edges.

The question that arises in this context is whether the rationality analysis of the previous sections still applies to the more general model \eqref{eq:SRL-corr} with correlated uncertainty.
The short answer to this is that the qualitative part of our analysis goes through essentially unchanged:

\begin{proposition}
\label{prop:corr}
Assume that \eqref{eq:SRL-corr} is run with a vanishing learning parameter $\temp(t)$.
Then, with notation as in the rest of the paper:
\begin{enumerate}
[1.]
\item
\label{itm:regret-corr}
\eqref{eq:SRL-corr} leads to no regret.
\item
\label{itm:dom-corr}
Dominated strategies become extinct under \eqref{eq:SRL-corr}.
\item
\label{itm:strict-corr}
Strict Nash equilibria are stochastically asymptotically stable under \eqref{eq:SRL-corr}.
\item
\label{itm:zerosum-corr}
In $2$-player zero-sum games with an interior equilibrium, the time-averaged process $\bar X(t) = t^{-1} \int_{0}^{t} X(s) \dd s$ converges to the game's set of Nash equilibria.
\end{enumerate}
Finally, properties \eqref{itm:dom-corr} and \eqref{itm:strict-corr} above hold even without the assumption $\temp(t)\to0$.
\end{proposition}

The proof of Proposition \ref{prop:corr} is similar to that of Theorems \ref{thm:consistent}, \ref{thm:dom}, \ref{thm:folk}, and Proposition \ref{prop:zero-sum}, so we only offer a short sketch in Appendix \ref{app:main}.
On the other hand, the quantitative bounds that we derived (e.g. for the player's regret or for the rate of elimination of dominated strategies) clearly cannot hold as stated because of the non-diagonal components of the quadratic variation matrix $\Sigma$.
It is still possible to calculate explicitly the corresponding expressions under the correlated noise model \eqref{eq:corr} but the end result is fairly messy so we do not present it here.

\subsection{Abrupt shocks}

Another important assumption underlying the noisy learning model \eqref{eq:SRL} (or its correlated version above) is that uncertainty enters the process in the form of a continuous, non-anticipative diffusion process.
This model is sufficiently wide to cover the case of noise processes whose aggregate impact scales as a function of the observation window but it does not account for discontinuous jump processes that represent very abrupt perturbations (such as defaults and stock price crashes, momentary link failures in data networks, etc.).

Processes of this type are usually modeled in probability theory as \emph{semimartingales}, i.e. stochastic processes that can be decomposed as a local martingale (e.g. a Wiener process) plus an adapted process with finite variation (e.g. any piecewise continuous martingale).
Given that the Itô integral can also be defined for semimartingales, it would also be possible to extend the scope of \eqref{eq:SRL} even further by replacing the Brownian noise terms $dW$ by an arbitary (adapted) semimartingale.
In this case however, Itô's lemma would also carry additional terms stemming from the discontinuous component of the noise process, thus complicating the analysis considerably;
because of this, we relegate a detailed discussion of this issue to a future paper.

\appendix
\section*{Appendix}

\section{Technical proofs}
\label{app:main}

\paragraph{Caveat}
In this appendix (and, in particular, in all applications of Itô's lemma), we will be treating the players' choice map $\choice$ as if it were $C^{1}$.
This is always the case if the players' penalty functions are steep (i.e. $\norm{dh(x_{n})}\to\infty$ whenever $x_{n}\to\bd(\strat)$), but not otherwise:
for instance, the closest-point projection map induced by the quadratic penalty function $h(x) = \frac{1}{2} \norm{x}^{2}$ is piecewise smooth but not $C^{1}$.
The use of Itô's lemma in the non-steep case is justified by the fact that $\choice$ is absolutely continuous \citep[Chap. 26]{Roc70} and the work of \citet[Theorem 2.1]{AJ05} who provide a nonsmooth variant of Itô's lemma that allows us to treat $h^{\ast}$ as if it were $C^{2}$ whenever it appears in a stochastic integral.

\subsection{Proof of Proposition \ref{prop:evolution}}

For simplicity, the player index $k$ is suppressed in this proof;
also, all summation indices are assumed to run over the (constant) support $\act'$ of $X(t)$.

By the \ac{KKT} conditions for the convex problem \eqref{eq:choice}, we get $Y_{\alpha} - \temp^{-1} \theta'(X_{\alpha}) = \zeta$ where $\zeta$ is the Lagrange multiplier corresponding to the probability constraint $\insum_{\alpha\in\supp(X)} X_{\alpha} = 1$.
Itô's formula then gives
\begin{equation}
d\zeta
	= \dd Y_\alpha
	+ \frac{\dot \temp}{\temp^{2}} \theta_{\alpha}' \dd t
	- \temp^{-1} \theta_{\alpha}'' \dd X_\alpha
	- \frac{1}{2} \temp^{-1} \theta_{\alpha}''' (dX_{\alpha})^{2}
\end{equation}
and hence:
\begin{equation}
\label{eq:aux1}
\dd X_\alpha
	= \frac{\temp}{\theta''_\alpha}
	\left[
		dY_{\alpha}
		+ \frac{\dot \temp}{\temp^{2}} \theta_{\alpha}' \dd t
		-  \frac{1}{2\temp} \theta_{\alpha}''' (dX_{\alpha})^{2}
		- \dd\zeta
	\right].
\end{equation}
Now, write $dX_{\alpha}$ in the general form:
\begin{equation}
\label{eq:aux2}
dX_{\alpha}
	= b_{\alpha} \dd t
	+ \insum_{\beta} c_{\alpha\beta} \dd W_\beta,
\end{equation}
where the drift and diffusion coefficients $b_{\alpha}$ and $c_{\alpha\beta}$ are to be determined.
To do so, recall that the Wiener processes $W_{\alpha}$ are independent (in the sense that $dW_{\alpha} \cdot dW_{\beta} = \delta_{\alpha\beta}$), so we get $(dX_{\alpha})^{2} = \insum_{\beta} c_{\alpha\beta}^{2} \dd t$;
therefore, summing \eqref{eq:aux1} over $\alpha\in\act'$ and solving for $d\zeta$ yields:
\begin{equation}
\frac{1}{\Theta''} \dd \zeta
	= \insum_{\alpha} \frac{1}{\theta_{\alpha}''}
	\left[
	\dd Y_{\alpha}
	+ \frac{\dot \temp}{\temp^{2}} \theta_{\alpha}' \dd t
	- \frac{1}{2\temp}\theta_{\alpha}''' \insum_{\beta} c_{\alpha\beta}^{2} \dd t
	\right ].
\end{equation}

Substituting this last expression in \eqref{eq:aux1} leads to \eqref{eq:evolution} with $U_{\alpha}^{2} = \insum_{\beta} c_{\alpha\beta}^{2}$, so we are left to show that $U_{\alpha}^{2}$ has the prescribed form.
To that end, by comparing the diffusion terms of \eqref{eq:evolution} and \eqref{eq:aux2}, we get 
\begin{equation}
c_{\alpha \beta}
	= \frac{\temp}{\theta_\alpha''}
	\left[
	\sigma_{\alpha} \delta_{\alpha\beta} - \frac{\Theta''}{\theta_\beta''}\sigma_{\beta}
	\right],
\end{equation}
and hence:
\begin{flalign}
U_{\alpha}^{2}
	=\insum_{\beta} c_{\alpha\beta}^{2}
	&= \left(\frac{\temp}{\theta_\alpha''} \right )^{2}
	\insum_{\beta} \left( \sigma_{\alpha} \delta_{\alpha\beta} - \Theta'' \sigma_{\beta}\big/\theta_{\beta}'' \right)^{2}
	\notag\\
	&= \left(\frac{\temp}{\theta_\alpha''} \right )^{2}
	\left[
		\sigma_{\alpha}^{2} \left(1 - \Theta''\big/\theta_{\alpha}''\right)^{2}
		+ \insum_{\beta\neq\alpha}
		\left( \Theta''\big/\theta_{\beta}'' \right )^{2} \sigma_{\beta}^{2}
	\right],
\end{flalign}
which concludes our derivation of \eqref{eq:evolution}.

To show that the support of $X(t)$ is piecewise constant on a dense open subset of $[0,\infty)$, fix some $\alpha\in\act$ and let $A = \{t:X_{\alpha}(t)>0\}$ so that $A^{c} = X_{\alpha}^{-1}(0)$.
Then, $A$ is open because $X(t)$ is continuous (a.s.), so it suffices to show that $A\cup\intr(A^{c})$ is dense in $[0,\infty)$;
however, this is trivially true because of the identity $\cl(A)\cup\intr(A^{c}) = [0,\infty)$.
Finally, if $\lim_{x\to0^{+}} \theta'(x) = -\infty$, standard convex analysis arguments show that the (necessarily unique) solution of \eqref{eq:choice} lies in the relative interior of $\strat$ \citep[Chap.~26]{Roc70}, so we conclude that $X(t)\in\intstrat$ for all $t\geq0$ by the well-posedness of \eqref{eq:SRL}.

\subsection{Proof of Theorem \ref{thm:consistent}}
Our proof hinges on the rate-adjusted Fenchel coupling
\begin{equation}
\label{eq:Hp}
H_{p}
	\equiv \frac{1}{\temp} \fench(p,\temp Y)
	= \frac{1}{\temp} \cdot\left[ h(p) + h^{\ast}(\temp Y) - \braket{\temp Y}{p} \right]
\end{equation}
between a benchmark strategy $p\in\strat$ and the generating process $\temp Y$.
Specifically, to begin with, the Itô formula of Lemma \ref{lem:Fenchel-gradient} gives:
\begin{flalign}
\label{eq:dH1}
dH_{p}
	&= - \frac{\dot\temp}{\temp} H_{p} \dd t
	+ \frac{1}{\temp} \braket{d(\temp Y)}{X-p}
	+ \frac{1}{2\temp} \insum_{\beta} \frac{\pd^{2} h^{\ast}}{\pd y_{\beta}^{2}} \temp^{2} \sigma_{\beta}^{2} \dd t
	\notag\\
	&= -\frac{\dot\temp}{\temp} H_{p} \dd t
	+ \frac{\dot\temp}{\temp} \braket{Y}{X - p} \dd t
	+ \braket{dY}{X-p}
	+ \frac{\temp}{2} \insum_{\beta} \frac{\pd^{2} h^{\ast}}{\pd y_{\beta}^{2}} \sigma_{\beta}^{2} \dd t,
\end{flalign}
so, by combining the definition of $H_{p}$ with \eqref{eq:USRL}, we get:
\begin{flalign}
\label{eq:dH2}
dH_{p}
	&= - \frac{\dot\temp}{\temp^{2}} \left[ h(p) - h(X) \right] dt
	+ \braket{\payv}{X-p} dt
	\notag\\
	&+ \insum_{\beta} (X_{\beta} - p_{\beta})\,\sigma_{\beta} \dd W_{\beta}
	+ \frac{\temp}{2} \insum_{\beta} \frac{\pd^{2}h^{\ast}}{\pd y_{\beta}^{2}} \sigma_{\beta}^{2} \dd t,
\end{flalign}
where we used the fact that $h^{\ast}(\temp Y) = \braket{\temp Y}{X} - h(X)$ in the first line.
Therefore, the player's cumulative regret for not playing $p$ up to time $t$ will be:
\begin{subequations}
\label{eq:reg}
\begin{flalign}
\int_{0}^{t} \braket{\payv(s)}{p - X(s)} \dd s
	&\label{eq:reg-Hdiff}
	= H_{p}(0) - H_{p}(t)
	\\
	&\label{eq:reg-temp}
	-\int_{0}^{t} \frac{\dot\temp(s)}{\temp^{2}(s)} \left[ h(p) - h(X(s)) \right] ds
	\\
	&\label{eq:reg-noise}
	+ \insum_{\beta} \int_{0}^{t} (X_{\beta}(s) - p) \,\sigma_{\beta}(s) \dd W_{\beta}(s)
	\\
	&\label{eq:reg-Ito}
	+ \frac{1}{2} \insum_{\beta} \int_{0}^{t} \temp(s) \frac{\pd^{2} h^{\ast}}{\pd^{2} y_{\beta}^{2}} \sigma_{\beta}^{2}(s) \dd s.
\end{flalign}
\end{subequations}

We now proceed to bound each term of \eqref{eq:reg} by a sublinear function:
\begin{enumerate}
[\itshape a\upshape)]

\item
Since $H_{p}\geq0$, the term \eqref{eq:reg-Hdiff} is bounded from above as follows:
\begin{equation}
\label{eq:reg-Hdiff1}
H_{p}(0)
	\leq \frac{h(p) + h^{\ast}(0)}{\temp(0)}
	= \frac{h(p) - \min_{x\in\strat} h(x)}{\temp(0)}
	\leq \frac{\depth}{\temp(0)},
\end{equation}
where we have used the fact that $Y(0) = 0$.
\item
For \eqref{eq:reg-temp},
we have $h(p) - h(X(s)) \leq \depth$ by definition;
hence, with $\dot\temp \leq 0$, we get:
\begin{equation}
\eqref{eq:reg-temp}
	\leq - \depth \int_{0}^{t} \frac{\dot\temp(s)}{\temp^{2}(s)} \dd s
	= \frac{\depth}{\temp(t)} - \frac{\depth}{\temp(0)}
	= o(t).
\end{equation}

\item
By Lemma \ref{lem:Wbound}, we readily obtain that the term \eqref{eq:reg-noise} is $\bigoh(\sigma_{\max} \sqrt{t \log \log t})$.

\item
Finally, for \eqref{eq:reg-Ito}, recall that the Hessian of $h^{\ast}(y)$ is equal to the inverse Hessian of $h(x)$, suitably restricted to the affine hull of the face of the simplex spanned by $x=\choice(y)$ \textendash\ see e.g. Chap.~26 in \cite{Roc70}.
Thus, given that $h$ is $K$-strongly convex, \eqref{eq:strong} readily yields $\pd^{2} h^{\ast}\big/\pd y_{\beta}^{2} \leq K^{-1}$, and hence:
\begin{equation}
\eqref{eq:reg-Ito}
	\leq \frac{\abs{\act}}{2K} \sigma_{\max}^{2} \int_{0}^{t} \temp(s) \dd s,
\end{equation}
where $\sigma_{\max}^{2} = \sup_{t\geq0} \max_{\beta}\sigma_{\beta}^{2}(t)$.
L'Hôpital's rule then yields $t^{-1} \int_{0}^{t} \temp(s) \dd s \sim \temp(t) \to 0$, so \eqref{eq:reg-Ito} is also sublinear in $t$.
\end{enumerate}
The bound \eqref{eq:reg-bound} then follows by combining all of the above.

\subsection{Proof of Propositions \ref{prop:dom-asymptotics} and \ref{prop:dom-mean}}

We begin with some groundwork that will be used in the proof of both propositions.
Specifically, let $X(t)$ be a solution of \eqref{eq:SRL} with $X(0) = x\in\strat$.
Then,
suppressing the player index $k$ for simplicity,
we obtain:
\begin{flalign}
\label{eq:dY0}
dY_{\alpha} - dY_{\beta}
	&= \left(\payv_{\alpha}(X) - \payv_{\beta}(X)\right) \dd t
	+ \sigma_{\alpha}(X) \dd W_{\alpha} - \sigma_{\beta}(X) \dd W_{\beta}
	\notag\\
	&\leq - m \dd t - d\xi,
\end{flalign}
where $\xi(t) = \int_{0}^{t} \sigma_{\beta}(X(s)) \dd W_{\beta}(s) - \int_{0}^{t} \sigma_{\alpha}(X(s)) \dd W_{\alpha}(s)$ and $m$ is defined as in the statement of the proposition.
We thus get:
\begin{equation}
\label{eq:Xrate0}
Y_{\alpha}(t) - Y_{\beta}(t)
	\leq Y_{\alpha}(0) - Y_{\beta}(0) - mt - \xi(t),
\end{equation}
so the \ac{KKT} conditions for the problem \eqref{eq:choice} give:%
\begin{equation}
\label{eq:Xrate1}
\theta'(X_{\alpha}(t)) - \theta'(X_{\beta}(t))
	\leq \temp(t)\cdot\left[Y_{\alpha}(0) - Y_{\beta}(0) - mt - \xi(t)\right]
\end{equation}
(recall that $\theta$ is steep so the solutions of \eqref{eq:choice} are interior).
Thus, with $\temp(t)$ decreasing, we finally get:
\begin{equation}
\label{eq:Xrate2}
\theta'(X_{\alpha}(t))
	\leq C - \temp(t)\cdot\left[mt + \xi(t)\right],
\end{equation}
where $C = \theta'(1) + \temp(0) \cdot \smallabs{Y_{\alpha}(0) - Y_{\beta}(0)}$ depends only on the initial conditions of \eqref{eq:SRL}.
This bound will be central in both proofs.

\begin{proof}[Proof of Proposition \ref{prop:dom-asymptotics}]
Let $\rho = [\xi,\xi]$ denote the quadratic variation of $\xi$ (cf. the proof of Lemma \ref{lem:Wbound}).
Then, the independence of $W_{\alpha}$ and $W_{\beta}$ yields $d\rho = d\xi\cdot d\xi = \big(\sigma_{\alpha}^{2} + \sigma_{\beta}^{2}\big) \dd t$ and hence:
\begin{equation}
\label{eq:covest2}
\rho(t)
	= \int_{0}^{t} \left[ \sigma_{\alpha}^{2}(X(s)) + \sigma_{\beta}^{2}(X(s)) \right] \dd s
	\leq 2\sigma_{\alpha\beta}^{2}\,t.
\end{equation}
Consequently, invoking the time-change theorem for martingales \citep{Oks07}, let $\wilde W$ be a Wie\-ner process such that $\xi(t) = \wilde W(\rho(t))$.
Then, if $\lim_{t\to\infty}\rho(t) = \infty$, by the law of the iterated logarithm, we get:
\begin{equation}
\label{eq:xi-log}
\xi(t)
	= \wilde W(\rho(t))
	\leq (1+\eps)\sqrt{2 \rho(t) \log \log \rho(t)}
	\leq 2(1+\eps) \sigma_{\alpha\beta} \sqrt{t \log \log t}
	\quad(a.s.),
\end{equation}
for all $\eps>0$ and for all large enough $t$.
Otherwise, if $\lim_{t\to\infty}\rho(t) < \infty$, $\wilde W(\rho(t))$ will be bounded for all $t\geq0$ (a.s.), so we trivially get $\xi(t) \geq 2(1+\eps) \sigma_{\alpha\beta}\sqrt{t \log \log t}$ for all large enough $t>0$.
The bound \eqref{eq:dom-rate} is then obtained by combining \eqref{eq:Xrate2} with \eqref{eq:xi-log}.

Likewise, combining \eqref{eq:Xrate2} and \eqref{eq:covest2}, we obtain:
\begin{flalign}
\label{eq:Xrate3}
\prob\left(X_{\alpha}(t) > \delta\right)
	&= \prob\left( \theta'(X_{\alpha}(t)) > \theta'(\delta) \right)
	\notag\\
	&\leq \prob\left( \wilde W(\rho(t)) < -mt + \frac{C - \theta'(\delta)}{\temp(t)}\right)
	\notag\\
	&= \frac{1}{2}
	\erfc\left[
	\frac{1}{\sqrt{2\rho(t)}}
	\left(mt - \frac{C - \theta'(\delta)}{\temp(t)}\right)
	\right].
\end{flalign}
Since $m>0$, Assumption \ref{asm:temp} for the rate of decay of $\temp(t)$ guarantees that $m t\temp(t) > C - \theta'(\delta)$ for large $t$;
\eqref{eq:dom-bound} then follows by substituting \eqref{eq:covest2} in \eqref{eq:Xrate3}.
\end{proof}

\begin{proof}[Proof of Proposition~\ref{prop:dom-mean}]
Let $Z_{\alpha\beta}$ denote the RHS of \eqref{eq:dY0}, viz.:
\begin{equation}
\label{eq:dZ}
dZ_{\alpha\beta}
	= m \dd t + \sigma_{\beta} \dd W_{\beta} - \sigma_{\alpha} \dd W_{\alpha}.
\end{equation}
Then, by \eqref{eq:Xrate2}, we readily obtain
\begin{equation}
\tau_{\delta}
	\leq \tilde\tau_{a}
	\equiv \inf\big\{t>0: Z_{\alpha\beta}(t) \leq a \big\},
\end{equation}
with $a = \temp^{-1} [C - \theta'(\delta)]_{+}$ (obviously, $\tau_{\delta} = 0$ if $C - \theta'(\delta) < 0$).
Since $Z_{\alpha\beta}$ is simply a Wiener process with drift $m$ (recall that the diffusion coefficients $\sigma$ are assumed constant), a standard argument based on Dynkin's formula yields $\ex[\tilde\tau_{a}] = a/m$ and \eqref{eq:dom-mean} follows.
\end{proof}

\subsection{Proof of Theorem \ref{thm:folk}}
Parts 1 and 2 of Theorem \ref{thm:folk} follow readily from Proposition \ref{prop:stability} \textendash\ simply note that the hypothesis of Proposition \ref{prop:stability} is satisfied in both cases.
As such, we only need to show that strict Nash equilibria are stochastically asymptotically stable under \eqref{eq:SRL}.

To that end, let $\eq = (\alpha_{1}^{\ast},\dotsc,\alpha_{N}^{\ast})$ be a strict equilibrium of $\game$ and let $\act_{k}^{\ast} \equiv \act_{k}\exclude\{\alpha_{k}^{\ast}\}$.
Moreover, for all $\alpha\in\act_{k}^{\ast}$, set
\begin{equation}
\label{eq:Zdef}
Z_{k\alpha}
	= \temp_{k} \left( Y_{k\alpha} - Y_{k\alpha_{k}^{\ast}} \right),
\end{equation}
so that $X(t)\to\eq$ if and only if $Z_{k\alpha}(t) \to -\infty$ for all $\alpha\in\act_{k}^{\ast}$, $k\in\play$ (cf. Proposition~\ref{prop:Fenchel}).

Now, fix some tolerance $\eps>0$ and a neighborhood $U_{0}$ of $\eq$ in $\strat$.
Since $\eq$ is a strict equilibrium of $\game$, there exist $m_{k}>0$ and a neighborhood $U\subseteq U_{0}$ of $\eq$ such that
\begin{equation}
\payv_{k\alpha_{k}^{\ast}}(x) - \payv_{k\alpha}(x)
	\geq m_{k}
	\quad
	\text{for all $x\in U$ and for all $\alpha\in\act_{k}^{\ast}$, $k\in\play$}.
\end{equation}
With this in mind, let $M>0$ be sufficiently large so that $X(t) \in U$ if $Z_{k\alpha}(t) \leq -M$ for all $\alpha\in\act_{k}^{\ast}$, $k\in\play$ (that such an $M$ exists is again a consequence of Proposition \ref{prop:Fenchel});
furthermore, with a fair degree of hindsight, assume also that
\begin{equation}
M
	> m_{k}^{-1} \temp_{k}(0) \sigma_{k,\max}^{2} \log(N/\eps)
	\quad
	\text{for all $k=1,\dotsc,N$,}
\end{equation}
where $\sigma_{k,\max}^{2} = \max_{\alpha,x} \sigma_{k\alpha}^{2}(x)$.
We will show that if $Z_{k\alpha}(0) \leq - 2M$, then $X(t) \in U$ for all $t\geq0$ and $Z_{k\alpha}(t)\to-\infty$ with probability at least $1-\eps$.

Indeed, assume that $Z_{k\alpha}(0) \leq -2M$ in \eqref{eq:Zdef} and define the escape time:
\begin{equation}
\label{eq:hitU}
\tau_{U}
	= \inf\{t>0: X(t) \notin U\}.
\end{equation}
The stochastic dynamics \eqref{eq:SRL} then give:
\begin{equation}
d\left( Y_{k\alpha} - Y_{k\alpha_{k}^{\ast}} \right)
	= \big[ \payv_{k\alpha} - \payv_{k\alpha_{k}^{\ast}} \big] \dd t
	+ \sigma_{k\alpha} \dd W_{k\alpha} - \sigma_{k\alpha_{k}^{\ast}} \dd W_{k\alpha_{k}^{\ast}},
\end{equation}
so, for all $t\leq \tau_{U}$, we will have:
\begin{flalign}
\label{eq:Z1}
Z_{k\alpha}(t)
	&= Z_{k\alpha}(0)
	+ \temp_{k}(t) \int_{0}^{t} \left[ \payv_{k\alpha}(X(s)) - \payv_{k\alpha_{k}^{\ast}}(X(s)) \right] \dd s
	+ \temp_{k}(t) \xi_{k}(t)
	\notag\\
	&\leq -2M - \temp_{k}(t) \left[ mt - \xi_{k}(t) \right],
\end{flalign}
where we have set
\begin{equation}
\label{eq:diffnoise-strict}
\xi_{k}(t)
	= \int_{0}^{t} \sigma_{k\alpha}(X(s)) \dd W_{k\alpha}(s) - \int_{0}^{t} \sigma_{k\alpha_{k}^{\ast}}(X(s)) \dd W_{k\alpha_{k}^{\ast}}(s).
\end{equation}

We will first show that $\prob(\tau_{U} < \infty) \leq \eps$.
To that end, note that the time-change theorem for martingales \cite[Cor.~8.5.4]{Oks07} provides a standard Wiener process $\wilde W_{k}(t)$ such that $\xi_{k}(t) = \wilde W_{k}(\rho_{k}(t))$ where $\rho_{k} = [\xi_{k},\xi_{k}]$ is the quadratic variation of $\xi_{k}$.
Then, from the fact that $\temp$ is nonincreasing, we conclude that $Z_{k\alpha}(t) \leq -M$ whenever $m_{k}t - \wilde W_{k}(\rho_{k}(t)) \geq - M/\temp_{k}(0)$.
Accordingly, with $\rho_{k}(t) \leq 2\sigma_{k,\max}^{2}t$ (cf. the proof of Proposition \ref{prop:dom-asymptotics}),
it suffices to show that the hitting time
\begin{equation}
\label{eq:hitline}
\tau_{0}
	= \inf\left\{
	t>0: \wilde W_{k}(t) = \frac{m_{k}t}{2\sigma_{k,\max}^{2}} + \frac{M}{\temp_{k}(0)}
	\;\text{for some $k\in\play$}
	\right\}
\end{equation}
is finite with probability not exceeding $\eps$.
However, if a trajectory of $\wilde W_{k}(t)$ has $\wilde W_{k}(t) \leq m_{k}t/(2\sigma_{k,\max}^{2}) + M/\temp_{k}(0)$ for all $t\geq0$, we will also have
\begin{equation}
\wilde W_{k}(\rho_{k}(t))
	\leq \frac{m_{k} \rho_{k}(t)}{2 \sigma_{k,\max}^{2}} + \frac{M}{\temp_{k}(0)}
	\leq m_{k}t + \frac{M}{\temp_{k}(0)},
\end{equation}
so $\tau_{U}$ must be infinite for every trajectory of $\wilde W = (\wilde W_{1},\dotsc,\wilde W_{N})$ with infinite $\tau_{0}$, i.e. $\prob(\tau_{U}<+\infty) \leq \prob(\tau_{0}<+\infty)$.
Thus, if we write $E_{k}$ for the event that $\wilde W_{k}(t) \geq m_{k}t/(2\sigma_{k,\max}^{2}) + M/\temp_{k}(0)$ for some finite $t\geq0$, the hitting-time analysis of \citet[p.~197]{KS98} for a Brownian motion with drift yields $\prob(E_{k}) = e^{-\lambda_{k} M}$ with $\lambda_{k} = m_{k} / (\temp_{k}(0) \sigma_{k,\max}^{2})$.
Thus, by construction of $M$, we obtain:
\begin{equation}
\prob(\tau_{0} < +\infty)
	= \prob\left( \union\nolimits_{k} E_{k} \right)
	\leq \insum_{k} \prob(E_{k})
	= \insum_{k} e^{-\lambda_{k}M}
	\leq \eps.
\end{equation}
Thus, by conditioning on the event $\tau_{U} = +\infty$ and invoking Assumption \ref{asm:temp} and Lemma \ref{lem:Wbound}, Eq.~\eqref{eq:Z1} finally yields
\begin{equation}
\label{eq:Z2}
Z_{k\alpha}(t)
	\leq -2M - \temp_{k}(t) \cdot \big[ mt - \xi_{k}(t) \big]
	\sim - \temp_{k}(t) \cdot mt
	\to-\infty
	\quad
	\text{(a.s.)}.
\end{equation}
We conclude that $Z_{k\alpha}(t)\to-\infty$ for all $\alpha\in\act_{k}^{\ast}$, $k\in\play$, so $\lim_{t\to\infty} X(t) = \eq$ (conditionally a.s.)
and our proof is complete.

\subsection{Proof of Proposition~\ref{prop:zero-sum}}
Let $\fench_{p}(t) = \insum_{k=1,2} \fench_{k}(p_{k},\temp_{k}(t) Y_{k}(t))$ where $p = (p_{1},p_{2})$ is an interior Nash equilibrium of $\game$.
Using the Itô formula of Lemma \ref{lem:Fenchel-gradient}, we get:
\begin{subequations}
\label{eq:dF}
\begin{flalign}
d\fench_{p}
	&= \sum_{k,\beta} (X_{k\beta} - p_{k\beta}) \dd(\temp_{k} Y_{k})
	+ \frac{1}{2} \sum_{k,\beta} \frac{\pd^{2} h_{k}^{\ast}}{\pd y_{k\beta}^{2}} \temp_{k}^{2} \sigma_{k\beta}^{2} \dd t
	\notag\\
	&\label{eq:dF-temp}
	= \sum_{k=1,2} \dot\temp_{k} \braket{Y_{k}}{X_{k} - p} dt
	\\
	&\label{eq:dF-game}
	+ \sum_{k=1,2} \temp_{k} \braket{\payv_{k}}{X_{k} - p_{k}} dt
	\\
	&\label{eq:dF-noise}
	+ \sum_{k,\beta} \temp_{k} (X_{k\beta} - p_{k\beta}) \sigma_{k\beta} \dd W_{k\beta}
	\\
	&\label{eq:dF-Ito}
	+ \frac{1}{2} \sum_{k,\beta} \frac{\pd^{2} h_{k}^{\ast}}{\pd y_{k\beta}^{2}} \temp_{k}^{2} \sigma_{k\beta}^{2} \dd t,
\end{flalign}
\end{subequations}
where we substituted $dY$ from \eqref{eq:SRL} to obtain \eqref{eq:dF-game} and \eqref{eq:dF-noise}.

We now claim that $\fench_{p}(t)$ grows sublinearly in $t$;
indeed:
\begin{enumerate}
[\itshape a\upshape)]
\item
Lemma \ref{lem:Wbound} shows that $Y(t) = \bigoh(t)$, so there exist $M_{k}>0$, $k=1,2$, such that
\begin{equation}
\sum_{k=1,2} \int_{0}^{t} \dot\temp_{k}(s) \braket{X_{k}(s) - p}{Y_{k}(s)} \dd s
	\leq \sum_{k=1,2} M_{k} \int_{0}^{t} \abs{\dot\temp_{k}(s)} s \dd s
	= o(t),
\end{equation}
where the sublinearity estimate follows from Assumption \ref{asm:temp} and the fact that $\temp_{k}(t)\to 0$:
\begin{equation}
\int_{0}^{t} \dot\temp_{k}(s) s \dd s
	= \temp_{k}(t) t - \int_{0}^{t} \temp_{k}(s) \dd s
	= o(t).
\end{equation}

\item
The term \eqref{eq:dF-game} is identically zero because $\game$ is zero-sum and $p$ is an interior equilibrium of $\game$:
\begin{multline}
\braket{\payv_{1}(X)}{X_{1} - p_{1}} + \braket{\payv_{2}(X)}{X_{2} - p_{2}}
	\\
	= \pay_{1}(X_{1},X_{2}) - \pay_{1}(p_{1},X_{2})
	+ \pay_{2}(X_{1},X_{2}) - \pay_{2}(X_{1},p_{2})
	= 0.
\end{multline}

\item
The term \eqref{eq:dF-noise} is sublinear (a.s.) on account of (the proof of) Lemma \ref{lem:Wbound}.

\item
Finally, for \eqref{eq:dF-Ito},
the same reasoning as in the proof of Theorem \ref{thm:consistent} yields:
\begin{equation}
\label{eq:dF-Ito2}
\frac{1}{2} \sum_{k,\beta} \int_{0}^{t} \frac{\pd^{2} h_{k}^{\ast}}{\pd y_{k\beta}^{2}} \temp_{k}^{2}(s) \sigma_{k\beta}^{2}(X(s)) \dd s
	\leq \frac{1}{2K} \insum_{k,\beta} \sigma_{k,\max}^{2} \int_{0}^{t} \temp_{k}^{2}(s) \dd s,
\end{equation}
with $\sigma_{k,\max}^{2} = \max_{\alpha\in\act_{k}} \max_{x\in\strat} \sigma_{k\alpha}^{2}(x)$ defined as in the proof of Theorem \ref{thm:folk}.
Since $\temp_{k}(t)\to0$, this last integral is also sublinear in $t$, as claimed.
\end{enumerate}

Assume now that $\limsup_{t\to\infty} \fench_{p}(t) \geq m$ for some $m > 0$ (otherwise we would have the stronger result $X(t)\to p$).
Then, Lemma \ref{lem:ydiff} gives $\smallabs{Y_{k\alpha}(t) - Y_{k\beta}(t)} = \bigoh(\fench_{p}(t)) = o(t)$, so our claim follows from Theorem \ref{thm:averages}.

\subsection{Proof of Theorem~\ref{thm:BRD}}
As we mentioned, the proof follows the ideas of \cite{HSV09}.
Specifically, from \eqref{eq:SRL}, we have:
\begin{flalign}
Y_{k}(t)
	&= Y_{k}(0) + \int_{0}^{t} \payv_{k}(X(s)) \dd s + \int_{0}^{t} \sigma_{k}(X(s)) \dd W_{k}(s)
	\notag\\
	&= t \payv_{k}(\bar X(t)) + Y_{k}(0) + \xi_{k}(t),
\end{flalign}
where we have set $\xi_{k\alpha}(t) = \int_{0}^{t} \sigma_{k\alpha}(X(s)) \dd W_{k\alpha}(s)$ and we have used the fact that $\game$ is a $2$-player game in order to carry the integral inside the argument of $\payv_{k}$.
Consequently, by the definition \eqref{eq:choice} of the players' regularized best response maps, $X_{k}(t) = \choice_{k}(\temp_{k}(t) Y_{k}(t))$ solves the (strictly) concave maximization problem:
\begin{equation}
\label{eq:choice1}
\begin{aligned}
\text{maximize}
	&\quad
	\smallbraket{\payv_{k}(\bar X(t))}{x_{k}}
	+ \frac{1}{t} \braket{Y_{k}(0) + \xi_{k}(t)}{x_{k}} - \frac{1}{t \temp_{k}(t)} h_{k}(x_{k}),
	\\
\text{subject to}
	&\quad
	x_{k} \in \strat_{k}.
\end{aligned}
\end{equation}
By Lemma \ref{lem:Wbound} and Assumption \ref{asm:temp}, the last two terms of \eqref{eq:choice1} vanish as $t\to0$.
Hence, by the maximum theorem of \citet[p.~116]{Ber97} applied to the parametric optimization problem \eqref{eq:choice1} with parameter $t$, it follows that $X_{k}(t)$ lies within a vanishing distance of $\argmax_{x_{k}\in\strat_{k}} \smallbraket{\payv_{k}(\bar X(t))}{x_{k}} \equiv \brep_{k}(\bar X(t))$.

On the other hand, differentiating $\bar X(t)$ yields:
\begin{equation}
\frac{d}{dt} \bar X(t)
	= t^{-1} X(t) - t^{-2} \int_{0}^{t} X(s) \dd s
	= t^{-1} \left[ X(t) - \bar X(t) \right],
\end{equation}
and, after changing time to $\tau = \log t$, the expression above becomes $\frac{d}{dt} \bar X(t) = X - \bar X$.
Combining all of the above, we conclude that $\bar X(t)$ tracks a perturbed version of the best reply dynamics \eqref{eq:BRD} in the sense of \citet[Def.~III]{BHS05}, and our assertion follows from Theorem 3.6 in the same paper.

\subsection{Correlated uncertainty}

We provide below a short sketch of the proof of Proposition \ref{prop:corr}, highlighting the points of departure from the proofs of Theorems \ref{thm:consistent}, \ref{thm:dom}, \ref{thm:folk}, and Proposition \ref{prop:zero-sum}.

First, for the no-regret analysis of \eqref{eq:SRL-corr}, note that \eqref{eq:reg} still holds essentially as stated, except for the term \eqref{eq:reg-Ito} which should be replaced by the quantity
\begin{equation}
\frac{1}{2} \sum_{\beta} \int_{0}^{t} \temp(s) \tr[\hess(h^{\ast}(Y(s)) \, \Sigma(s)] \dd s.
\end{equation}
Since the matrices $\hess(h^{\ast})$ and $\Sigma$ are both bounded in norm (by the strong convexity of $h$ and by assumption, respectively), the rest of the proof of Theorem \ref{thm:consistent} goes through exactly as in the uncorrelated case, except that $\sigma_{\max}$ must now be replaced by $\sup_{t} \norm{\Sigma(t)}$.

For the elimination of dominated strategies and the stochastic asymptotic stability of strict equilibria, we first need to replace all terms of the form $\sigma_{k\alpha} d W_{k\alpha}$ in \eqref{eq:diffnoise} and \eqref{eq:diffnoise-strict} by $dZ_{k\alpha}$.
Then, it suffices to note that the rest of the proof of Theorems \ref{thm:dom} and \ref{thm:folk} only relies on the fact that the quadratic variation of $\xi$ satisfies an inequality of the form $d[\xi,\xi] \leq M \dd t$ for some $M>0$ (a consequence of the fact that $\Sigma$ is assumed bounded).

Finally, for our result on time-averages, the only nontrivial change in the proof of Proposition \ref{prop:zero-sum} is that the term $\sum_{\beta} \frac{\pd^{2}h^{\ast}}{\pd y_{\beta}^{2}} \sigma_{\beta}^{2} \dd t$ that appears in \eqref{eq:dF} must be replaced by the correlated variant $\tr[\hess(h^{\ast}(Y(s)) \, \Sigma(s)] \dd s$.
This term can be bounded as in the no-regret case above, so \eqref{eq:dF-Ito2} still holds (but with a different constant than $\sigma_{\max}^{2}/K$);
the rest of the proof then goes through as in the uncorrelated setting.


\section{Auxiliary results}
\label{app:auxiliary}

\subsection{Properties of the Fenchel coupling}
\label{app:Fenchel}

In this appendix, we collect some basic properties of the Fenchel coupling and regularized best response maps.
To state them, let $\simplex$ denote the $d$-dimensional simplex of $\R^{n}$ and let $h\from\simplex\to\R$ be a penalty function on $\simplex$ (i.e. $h$ is continuous and strongly convex on $\simplex$ and smooth on the relative interior of any face of $\simplex$).
Also, recall that the \emph{regularized best response map} $\choice\from\R^{n}\to\simplex$ induced by $h$ is given by
\begin{equation}
\choice(y)
	= \argmax_{x\in\simplex} \{ \braket{y}{x} - h(x)\},
	\quad
	y\in \R^{n},
\end{equation}
while the \emph{Fenchel coupling} between $p\in\simplex$ and $y\in\R^{n}$ is defined as:
\begin{equation}
\fench(p,y)
	= h(p) + h^{\ast}(y) - \braket{y}{p},
\end{equation}
where
\begin{equation}
h^{\ast}(y)
	= \max_{x\in\simplex} \{ \braket{y}{x} - h(x) \}
\end{equation}
denotes the convex conjugate of $h$.
We then have:

\begin{proposition}
\label{prop:Fenchel}
Let $h\from\simplex\to\R$ be a penalty function on $\simplex$ with strong convexity constant $K$.
Then:
\begin{enumerate}
[\indent\upshape(1)]
\item
$\choice$ is $1/K$-Lipschitz and $\choice(y) = dh^{\ast}(y)$.

\item
If $y_{\alpha} - y_{\beta} \to -\infty$ for some $\beta\neq\alpha$, then $\choice_{\alpha}(y) \to 0$.

\item
$\fench(p,y) \geq \frac{1}{2} \norm{\choice(y) - p}^{2}$ for all $p\in\simplex$ and for all $y$;
in particular, $\fench(p,y) \geq 0$ with equality if and only if $p=\choice(y)$.

\item
If $\fench(p,y_{j})\to+\infty$ for some sequence $y_{j}\in \R^{n}$, the sequence $x_{j} = \choice(y_{j})$ converges to the union of faces of $\simplex$ that do not contain $p$;
more precisely: $\liminf_{j\to\infty} \{x_{j,\alpha}:\alpha\in\supp(p)\} = 0$.
\end{enumerate}
\end{proposition}

\begin{proof}
The first part of Proposition \ref{prop:Fenchel} is well known \citep[see e.g.][]{Nes09,SS11}.
The rest of our claims are due to \cite{MS16};
we only provide a proof below for completeness and ease of reference.


\paragraph{Part 2}
Let $y_{j}$ be a sequence in $\R^{n}$ such that $y_{j,\alpha} - y_{j,\beta}\to-\infty$, set $x_{j} = \choice(y_{j})$ and assume (by descending to a subsequence if necessary) that $x_{j,\alpha} > \eps >0$ for all $n$.
By definition, we have $\braket{y_{j}}{x_{j}} - h(x_{j}) \geq \braket{y_{j}}{x'} - h(x')$ for all $x'\in\simplex$, so if we set $x_{j}' = x_{j} + \eps(\bvec_{\beta} - \bvec_{\alpha})$, we readily obtain
\begin{equation}
\label{eq:Qcomp2}
\eps (y_{j,\alpha} - y_{j,\beta})
	\geq h(x_{j}) - h(x_{j}')
	\geq - (h_{\max} - h_{\min}).
\end{equation}
This contradicts our original assumption that $y_{j,\alpha} - y_{j,\beta} \to -\infty$;
since $\simplex$ is compact, we get $\eq_{\alpha} = 0$ for any limit point $\eq$ of $x_{j}$, i.e. $\choice_{\alpha}(y_{j})\to0$.

\paragraph{Part 3}
Let $x=\choice(y)$ so that $h^{\ast}(y) = \braket{y}{x} - h(x)$ and $\fench(p,y) = h(p) - h(x) - \braket{y}{p-x}$.
Since $x=dh^{\ast}(y)$ by Part 1, we will also have
\begin{equation}
\label{eq:subgradient}
\braket{y}{x'-x}
	\leq h(x') - h(x)
	\quad
	\text{for all $x'\in\simplex$,}
\end{equation}
by standard convex analysis arguments \citep[Chap.~26]{Roc70}.
On the other hand, letting $z = p-x$, the definition \eqref{eq:strong} of strong convexity yields:
\begin{equation}
\label{eq:Dbound1}
h(x+tz)
	\leq t h(p) + (1-t) h(x) - \tfrac{1}{2} K t(1-t) \norm{z}^{2}
	\quad
	\text{for all $t\in(0,1)$.}
\end{equation}
Hence, combining \eqref{eq:Dbound1} with \eqref{eq:subgradient} for $x' = x + tz$, we get
\begin{equation}
\label{eq:Dbound2}
h(x) + \braket{y}{tz}
	\leq t h(p) +(1-t) h(x) - \tfrac{1}{2} K t(1-t) \norm{z}^{2},
\end{equation}
and, after rearranging and dividing by $t$, we obtain
\begin{equation}
\label{eq:Dbound3}
\fench(p,y)
	= h(p) - h(x) - \braket{y}{z}
	\geq \tfrac{1}{2} K(1-t) \norm{z}^{2}.
\end{equation}
Our claim then follows by letting $t\to0^{+}$ in \eqref{eq:Dbound3}.

\paragraph{Part 4}
Let $\Lambda_{p}$ denote the union of faces of $\simplex$ that do not contain $p$ and set $\simplex_{p} \equiv \simplex\exclude\Lambda_{p}$.
Then, if $x=\choice(y) \in \simplex_{p}$ for some $y\in \R^{n}$, the function $h(x+t(p-x))$ will be finite and smooth for all $t$ in a neighborhood of $0$ (simply note that $x+ t(p-x)$ lies in the relative interior of the same face of $\simplex$ for small $t$).
This implies that $h$ admits a two-sided derivative at $x$ along $p-x$, so we will also have $\braket{y}{p-x} = h'(x;p-x)<+\infty$ by Theorem 23.5 in \cite{Roc70};
in particular, this shows that $\fench(p,y) = h(p) - h(x) - h'(x;p-x) < +\infty$ whenever $x=\choice(y)\in\simplex_{p}$.

Assume now ad absurdum that $x_{j} = \choice(y_{j})$ has a limit point in $\simplex_{p}$, so there exists a compact neighborhood $K$ of $p$ in $\simplex_{p}$ which is visited infinitely often by $x_{j}$.
By continuity and compactness, this implies that there exists some $M>0$ such that $\fench(p,y_{j}) = h(p) - h(x_{j}) - h'(x_{j};p-x_{j}) < M$ infinitely often, a contradiction to our original assumption that $\fench(p,y_{j})\to+\infty$.
We conclude that every limit point of $x_{j}$ lies in $\simplex_{p}$, as claimed.
\end{proof}

For the proof of Proposition \ref{prop:zero-sum}, we also require the following lemma:

\begin{lemma}
\label{lem:ydiff}
Let $p\in\intsimplex$ and assume that the sequence $y_{j}$ has $\fench(p,y_{j}) \geq m$ for some $m>0$.
Then, for all $\alpha,\beta$, there exists some $K>0$ such that $\smallabs{y_{j,\alpha} - y_{j,\beta}} \leq K \fench(p,y_{j})$.
\end{lemma}

\begin{proof}
By descending to a subsequence if necessary, assume ad absurdum that $y_{j,\alpha} - y_{j,\beta} \geq K_{j} \fench(p,y_{j})$ for some increasing sequence $K_{j}\to+\infty$.
Then, by relabeling indices (and passing to a finer subsequence if needed), we may assume that
\begin{inparaenum}
[\itshape a\upshape)]
\item
$y_{j,\alpha} \geq y_{j,\gamma} \geq y_{j,\beta}$ for all $\gamma=1,\dotsc,d$;
and
\item
the index set $\{1,\dotsc,d\}$ can be decomposed into two nonempty sets, $\act_{0}$ and $\act_{\infty}$, such that $y_{j,\alpha} - y_{j,\gamma} \leq K' \fench(p,y_{j})$ for some $K'>0$ and for all $\gamma\in\act_{0}$, while $y_{j,\alpha} - y_{j,\gamma} > K_{j}' \fench(p,y_{j})$ for some $K_{j}'\to+\infty$ and for all $\gamma\in\act_{\infty}$.
\end{inparaenum}
Thus, letting $x_{j} = \choice(y_{j})$, we get:
\begin{flalign}
\label{eq:ycomp}
\braket{y_{j}}{p - x_{j}}
	&= \insum_{\gamma=1}^{n} y_{j,\gamma} (p_{\gamma} - x_{j,\gamma})
	= \insum_{\gamma=1}^{n} (y_{j,\gamma} - y_{j,\alpha}) (p_{\gamma} - x_{j,\gamma})
	\notag\\
	&= \sum_{\gamma\in\act_{0}} (y_{j,\gamma} - y_{j,\alpha}) (p_{\gamma} - x_{j,\gamma})
	+ \sum_{\gamma\in\act_{\infty}} (y_{j,\gamma} - y_{j,\alpha}) (p_{\gamma} - x_{j,\gamma}).
\end{flalign}
By construction, the first sum is bounded in absolute value by $2dK'\fench(p,y_{j})$.
As for the second sum, since $\fench(p,y_{j})\geq M > 0$, we will also have $y_{j,\alpha} - y_{j,\gamma} \to +\infty$, so $x_{j,\gamma}\to0$ by Proposition \ref{prop:Fenchel};
with $p\in\intsimplex$, we then get $\liminf_{j} (p_{\gamma} - x_{j,\gamma}) \geq \delta$ for some $\delta>0$, so every $\act_{\infty}$ summand in \eqref{eq:ycomp} is eventually bounded from above by $-\delta K_{j}'\fench(p,y_{j})$.
This shows that $\braket{y_{j}}{p-x_{j}} \leq - K_{j}''\fench(p,y_{j})$ for some positive sequence $K_{j}''\to+\infty$;
however, $\braket{y_{j}}{p - x_{j}} = h(p) - h(x_{j}) - \fench(p,y_{j})$, a contradiction.
\end{proof}

\subsection{Stochastic analysis}
\label{app:probability}

From a stochastic analysis standpoint, the importance of the Fenchel coupling is captured by the following calculation:

\begin{lemma}
\label{lem:Fenchel-gradient}
Let $dY_{\alpha} = \mu_{\alpha} \dd t + \insum_{\beta} \sigma_{\alpha\beta} \dd W_{\beta}$ be an Itô process with values in $\R^{n}$ and set $X(t) = \choice(Y(t))$.
Then, for all $p\in\simplex$, we have:
\begin{equation}
\label{eq:Fenchel-gradient}
d\fench(p,Y)
	= \insum_{\beta} (X_{\beta} - p_{\beta}) \dd Y_{\beta}
	+ \frac{1}{2} \insum_{\beta} \frac{\pd^{2} h^{\ast}}{\pd y_{\beta}^{2}} \sigma_{\beta}^{2} \dd t.
\end{equation}
\end{lemma}

\begin{proof}
Let $H(t) = \fench(p,Y(t))$, $t\geq0$.
Then, Itô's formula yields:
\begin{flalign}
dH
	&= \insum_{\beta} \frac{\pd\fench}{\pd y_{\beta}} dY_{\beta}
	+ \frac{1}{2} \insum_{\beta,\gamma} \frac{\pd^{2} \fench}{\pd y_{\beta} \pd y_{\gamma}} \dd Y_{\beta}\cdot dY_{\gamma}
	\notag\\
	&= \insum_{\beta} \left(\frac{\pd h^{\ast}}{\pd y_{\beta}} - x_{\beta}\right) \dd Y_{\beta}
	+ \frac{1}{2} \insum_{\beta,\gamma} \frac{\pd^{2} h^{\ast}}{\pd y_{\beta}^{2}} \sigma_{\beta} \sigma_{\gamma} \delta_{\beta\gamma}\dd t
	\notag\\
	&= \insum_{\beta} (X_{\beta} - p_{\beta}) \dd Y_{\beta}
	+ \frac{1}{2} \insum_{\beta} \frac{\pd^{2}h^{\ast}}{\pd y_{\beta}^{2}} \sigma_{\beta}^{2} \dd t,
\end{flalign}
where we used the definition of $\fench$ and Proposition \ref{prop:Fenchel} in the second and third lines.
\end{proof}

Finally, Lemma \ref{lem:Wbound} is a useful order comparison result that we employ throughout:

\begin{lemma}
\label{lem:Wbound}
Let $W(t) = (W_{1}(t),\dotsc,W_{n}(t))$, $t\geq0$, be a Wiener processes in $\R^{n}$ and let $Z(t)$ be a bounded, continuous process in $\R^{n}$.
Then, for every continuous function $f\from[0,\infty)\to(0,\infty)$, we have
\begin{equation}
\label{eq:Wbound}
f(t) + \int_{0}^{t} Z(s) \cdot dW(s)
	\sim f(t)
	\quad
	\text{as $t\to\infty$ \textup(a.s.\textup),}
\end{equation}
whenever
$\lim_{t\to\infty} \left(t\log\log t\right)^{-1/2} f(t) = +\infty$.
\end{lemma}

\begin{proof}
Let $\xi(t) = \int_{0}^{t} Z(s) \cdot dW(s) = \sum_{\alpha=1}^{n} \int_{0}^{t} Z_{\alpha}(s) \dd W_{\alpha}(s)$.
Then, the quadratic variation $\rho = [\xi,\xi]$ of $\xi$ satisfies:
\begin{equation}
\label{eq:covest1}
d[\xi,\xi]
	= d\xi \cdot d\xi
	= \insum_{\alpha=1}^{n} Z_{\alpha} Z_{\beta} \delta_{\alpha\beta} \dd t
	\leq M \dd t,
\end{equation}
where $M = \sup_{t\geq0} \norm{Z(t)}^{2} < +\infty$ (recall that $Z(t)$ is bounded by assumption).
On the other hand, by the time-change theorem for martingales \citep[Cor.~8.5.4]{Oks07}, there exists a Wiener process $\wilde W(t)$ such that $\xi(t) = \wilde W(\rho(t))$, and hence:
\begin{equation}
\frac{f(t) + \xi(t)}{f(t)}
	= 1 + \frac{\wilde W(\rho(t))}{f(t)}.
\end{equation}
Obviously, if $\lim_{t\to\infty} \rho(t) < \infty$,
we will have $\limsup_{t\to\infty} \smallabs{\wilde W(\rho(t))} < \infty$ (a.s.) so there is nothing to show.
Otherwise, if $\lim_{t\to\infty} \rho(t) = \infty$, the quadratic variation bound \eqref{eq:covest1} and the law of the iterated logarithm yield:
\begin{equation}
\frac{\big\vert\wilde W(\rho(t))\big\vert}{f(t)}
	\leq \frac{\big\vert\wilde W(\rho(t)) \big\vert}{\sqrt{2\rho(t) \log \log \rho(t)}}
	\times \frac{\sqrt{2Mt \log \log Mt}}{f(t)}
	\to 0
	\quad
	\text{as $t\to\infty$,}
\end{equation}
and our claim follows.
\end{proof}
\bibliographystyle{elsarticle-harv}
\bibliography{IEEEabrv,Bibliography}

\end{document}